\documentclass[10pt]{amsart}
\usepackage{amsmath}
\usepackage{amsfonts}
\usepackage{amssymb}
\usepackage{graphicx}
\usepackage{amsthm,graphicx,color,yfonts}
\usepackage{pdfsync}
\usepackage{epstopdf}
\usepackage{epsfig}
\usepackage{marginnote} 
\usepackage{esint} 

\usepackage[colorlinks=true]{hyperref}
\hypersetup{linkcolor=red,citecolor=blue,filecolor=dullmagenta,urlcolor=darkblue} 

\numberwithin{equation}{section}

\theoremstyle{plain}
\newtheorem{theorem}{Theorem}[section]

\newtheorem{lemma}[theorem]{Lemma}
\newtheorem*{de-lemma}{Lemma}

\theoremstyle{remark}

\theoremstyle{definition}

\newcommand{\dd}{\mathrm{d}}
\newcommand{\R}{\mathbb{R}}

\newcommand{\n}{\textbf{\em n}}
\newcommand{\ve}{\epsilon}

\begin{document}

\title{Symmetry breaking and restoration  in the Ginzburg-Landau model of nematic liquid crystals}

\author{Marcel G. Clerc}
\address{Departamento de F\'{i}sica, FCFM, Universidad de Chile, Casilla 487-3, Santiago,
Chile.}
\email{marcel@dfi.uchile.cl}
\thanks{M.G. Clerc was partially supported by Fondecyt 1150507.}


\author{Micha{\l } Kowalczyk}
\address{Departamento de Ingenier\'{\i}a Matem\'atica and Centro
de Modelamiento Matem\'atico (UMI 2807 CNRS), Universidad de Chile, Casilla
170 Correo 3, Santiago, Chile.}
\email {kowalczy@dim.uchile.cl}
\thanks{M. Kowalczyk was partially supported by Chilean research grants Fondecyt 1130126 and 1170164, Fondo Basal CMM-Chile}

\author{Panayotis Smyrnelis}
\address{Centro
de Modelamiento Matem\'atico (UMI 2807 CNRS), Universidad de Chile, Casilla
170 Correo 3, Santiago, Chile.}
\email{psmyrnelis@dim.uchile.cl}
\thanks{P. Smyrnelis was partially supported by Fondo Basal CMM-Chile and Fondecyt postdoctoral grant 3160055}


\begin{abstract}
In this paper we study qualitative properties of global minimizers of the Ginzburg-Landau energy which describes light-matter interaction  in the theory of nematic liquid crystals near  the Friedrichs transition. This model is depends on  two parameters: $\epsilon>0$ which is small and represents the coherence scale of the system and $a\geq 0$ which represents the intensity of the applied laser light. In  particular we are interested in the phenomenon of symmetry breaking as $a$ and $\epsilon$ vary. We show that when  $a=0$  the global minimizer is radially symmetric and unique and that its  symmetry is instantly broken  as $a>0$ and then restored for sufficiently large values of $a$. Symmetry breaking is associated with the presence of a new type of topological defect which we named  the {\it shadow vortex}. The symmetry breaking scenario is a rigorous confirmation of experimental and numerical results obtained earlier in \cite{clerc2}. 
\end{abstract}

\maketitle

\section{Introduction}

In a suitable experimental set up \cite{LCLV-Vortex2015,PhysRevLett.111.093902, clerc2, Barboza2012, {Barboza2015}, Barboza2015A} involving a liquid crystal sample, a laser and a photoconducting cell one can observe light defects  such as kinks, domain walls and vortices. A concrete example of formation of optical vortices is presented in \cite{clerc2}. To describe this phenomenon starting from the classical Oseen-Frank energy near the Fredrichs transition one can reduce the problem to considering  the Ginzburg-Landau energy as it was explained in \cite{panayotis_1}.    
After some transformations involving scaling to non dimensional variables the latter energy takes form: 
\begin{equation}
\label{funct 0}
E(u)=\int_{\R^2}\frac{1}{2}|\nabla u|^2-\frac{1}{2\epsilon^2}\mu(x)|u|^2+\frac{1}{4\epsilon^2}|u|^4-\frac{a}{\epsilon} f(x)\cdot u,
\end{equation}
where $u=(u_1,u_2)\in  H^1(\R^2,\R^2)$ and $\epsilon>0$, $a\geq 0$ are real parameters. In the physical context described in \cite{clerc2} the functions $\mu$ and $f$ are specific:
\[
\mu(x)=e^{\,-|x|^2}-\chi, \qquad \mbox{with some}\ \chi\in (0,1), \qquad f(x)=-\frac{1}{2}\nabla \mu(x).
\]
Physically the order parameter  $u$  represents the intensity of  light induced by the  interaction between the  laser beam of Gaussian profile (given by $\mu$) and the nematic liquid crystal sample with the photoconducting cell mounted on top of it. This cell generates  electric field whose small, vertical component is described  above by $f$. The parameter $a$ is non dimensional and characterizes the intensity of the laser beam.  The two dimensional model (\ref{funct 0}) shows an excellent agreement with  experiments performed with physical parameters near the Fredrichs transition \cite{clerc2}.

All our results hold under more general hypothesis on $\mu$ and $f$ which we will state now. 
We suppose that $\mu \in C^\infty(\R^2,\R)$ is radial i.e. $\mu(x)=\mu_{\mathrm{rad}}(|x|)$, with $\mu_{\mathrm{rad}} \in C^\infty(\R,\R)$ such that $\mu_{\mathrm{rad}}$ has an even extension to the whole real line. 
We take  $f=(f_1,f_2)\in C^\infty(\R^2,\R^2)$  also to be  radial i.e. $f(x)= f_{\mathrm{rad}}(|x|)\frac{x}{|x|}$, with $f_{\mathrm{rad}} \in C^\infty(\R,\R)$ such that $f_{\mathrm{rad}}$ has an odd extension to the whole real line.
In addition we assume that 
\begin{equation}\label{hyp2}
 \begin{cases}
   \mu \in L^\infty(\R^2,\R),\ \mu_{\mathrm{rad}}'<0 \text{ in }(0,\infty), \text{ and $ \mu_{\mathrm{rad}}(\rho)=0$ for a unique $\rho>0$},\medskip \\
  \text{$f \in L^1(\R^2,\R^2)\cap L^\infty(\R^2,\R^2)$, and $ f_{\mathrm{rad}}>0$ on $(0,\infty)$}.
 \end{cases}
\end{equation}

The Euler-Lagrange equation of $E$ is
\begin{equation}\label{ode}
\ve^2 \Delta u+\mu(x) u-|u|^2u+\ve a f(x)=0,\qquad x\in \R^2.
\end{equation} 
We also write its weak formulation:
\begin{equation}\label{euler}
\int_{\R^2} -\epsilon^2 \sum_{j=1,2}\nabla u_j\cdot \nabla \psi_j+\mu u\cdot \psi-|u|^2u\cdot \psi+\epsilon a f\cdot\psi=0,\qquad  \forall \psi \in H^1(\R^2,\R^2),
\end{equation}
where $\cdot$ denotes the inner product in $\R^2$.
Note that due to the radial symmetry of $\mu$ and $f$, the energy \eqref{funct 0} and equation 
\eqref{ode} are invariant under the transformations $v(x)\mapsto g^{-1}v(gx)$, $\forall g \in O(2)$.

Our main purpose in this paper is to study qualitative properties of the global minimizers of $E$ as the parameters $a$ and $\epsilon$ vary. In general we will assume that $\epsilon>0$ is small and $a\geq 0$ is bounded uniformly in $\epsilon$.  
As we will see critical phenomena such as symmetry breaking and restoration, which are the focus of this paper, occur along curves of the form $a=a(\epsilon)$ in the $(\epsilon, a)$ plane.   

In Lemma \ref{lem exist min} we show that under the above assumptions there exists a global minimizer $v$ of $E$ in $H^1(\R^2,\R^2)$, namely  that $E(v)=\min_{H^1(\R^2,\R^2)} E$. 
In addition, we show that $v$ is a classical solution of \eqref{ode}. Some basic properties of the global minimizer are stated in: 
\begin{theorem}\label{theorem 1}
Let $v_{\epsilon, a}$ be the global minimizer of $E$, let $a\geq 0$ be bounded (possibly dependent  on $\epsilon$), let $\rho>0$ be the zero of $\mu_{\mathrm{rad}}$ and let $\mu_1:=\mu_{\mathrm{rad}}'(\rho)<0$. The following  statements hold:
\begin{itemize}
\item[(i)] Let $\Omega\subset D(0;\rho)$ be an open set such that $v_{\epsilon, a}\neq 0$ on $\Omega$, for every $\epsilon\ll 1$.
Then $|v_{\epsilon, a}|\to \sqrt{\mu}$ in $C^0_{\mathrm{loc}}(\Omega)$. 
\item[(ii)] For every $\xi=\rho e^{i\theta}$, we consider the local coordinates $s=(s_1,s_2)$ in the basis $(e^{i\theta},i e^{i\theta})$, and the 
rescaled minimizers:
\[
w_{\epsilon,a}(s)= 2^{-1/2}(-\mu_1\ve)^{-1/3} v_{\epsilon,a}\Big( \xi+\ve^{2/3} \frac{s}{(-\mu_1)^{1/3}}\Big).
\]
As $\ve\to 0$, the function $w_{\epsilon, a}$ converges in $C^2_{\mathrm{loc}}(\R^2,\R^2)$ up to subsequence, to a bounded in the half-planes $[s_0,\infty)\times \R$ solution of
\begin{equation}\label{pain}
\Delta y(s)-s_1 y(s)-2|y(s)|^2y(s)-\alpha=0, \qquad \forall s=(s_1,s_2)\in \R^2,
\end{equation}
with $\alpha=\lim_{\epsilon\to 0}\frac{a(\epsilon) f(\xi)}{\sqrt{2}\mu_1}$. 
\item[(iii)] For every $r_0>\rho$, we have $\lim_{\epsilon\to 0}\frac{v_{\epsilon,a}((r_0+t\epsilon)e^{i\theta})}{\epsilon}=
-\frac{a_0}{\mu_{\mathrm{rad}}(r_0)}f(r_0 e^{i\theta})$ uniformly when $t$ remains bounded and $\theta \in\R$, with $a_0:=\lim_{\epsilon\to 0}a(\epsilon)$.
\end{itemize}
\end{theorem}


Looking at the energy $E$ it is evident that as $\epsilon\to 0$ the modulus of the global minimizer $|v_{\epsilon,a}|$ should approach a nonnegative 
root of the polynomial
\[
-\mu_{\mathrm{rad}}(|x|)y+y^3-a\epsilon f_{\mathrm{rad}}(|x|)=0,
\] 
or in other words $|v_{\epsilon, a}|\to \sqrt{\mu^+}$ as $\epsilon\to 0$ in some, perhaps weak, sense.  We observe for instance  that as a  corollary 
of Theorem \ref{theorem 1} (i) and Theorem \ref{theoremz} (ii) below we obtain that when $a=o(\ve|\ln \ve|)$ we have convergence  in 
$C^0_{\mathrm{loc}}(D(0; \rho))$. Because of the analogy between the functional $E$  and the Gross-Pitaevskii functional in theory of Bose-Einstein condensates we will call $\sqrt{\mu^+}$ the Thomas-Fermi limit  of the global minimizer (we will comment more on this connection later on). 
Theorem \ref{theorem 1} gives account on how non smoothness of  the limit  of $|v_{\epsilon, a}|$  is mediated near the circumference $|x|=\rho$, where 
$\mu$ changes sign, through the solution of (\ref{pain}). This  equation is a natural generalization of the second Painlev\'e equation 
\begin{equation}
\label{pain 1d}
y''-sy-2y^3-\alpha=0, \qquad s\in \R.
\end{equation}
In \cite{panayotis_1} we showed that this last equation plays an analogous role in the one dimensional, scalar version of the energy $E$:
\[
E(u,\R)=\int_{\R}\frac{\epsilon}{2}|u_x|^2-\frac{1}{2\epsilon}\mu(x)u^2+\frac{1}{4\epsilon}|u|^4-a f(x)u
\] 
where $\mu$ and $f$ are scalar functions satisfying similar hypothesis to those we have described above.  In this case the Thomas-Fermi limit of the global minimizer is simply $\sqrt{\mu^+(x)}$, which is non differentiable at  the points $x=\pm \xi$ which are the zeros of the even function $\mu$. Near these two points a rescaled version of the global minimizer approaches a solution of (\ref{pain 1d}) similarly as it is described in Theorem  \ref{theorem 1} (ii). It is very important to realize that not every solution of (\ref{pain 1d}) can  serve as the limit, actually there are only two such solutions: $y^+$ which is positive, decays to $0$ at $+\infty$  and grows like $\sqrt{|s|}$ at $-\infty$ and $y^-$ which is sign changing, has similar asymptotic behavior at $+\infty$ but $y^-(s)\sim -\sqrt{|s|}$ near $-\infty$. To show existence of $y^\pm$ is quite nontrivial and for proofs we refer to \cite{2005math.ph...8062C}, 
\cite{MR555581}, \cite{troy1}.
Moreover these two solutions are minimal.
To explain what this means  we go back  to the present problem since in our case the limiting solutions of (\ref{pain}) are necessarily minimal as well. Let 
\[
E_{\mathrm{P_{II}}}(u, A)=\int_A \left[ \frac{1}{2} |\nabla u|^2 +\frac{1}{2}  s_1 |u|^2 +\frac{1}{2} |u|^4+\alpha\cdot u\right].
\]
By definition  a  solution of (\ref{pain}) is minimal if
\begin{equation}\label{minnn}
E_{\mathrm{P_{II}}}(y, \mathrm{supp}\, \phi)\leq E_{\mathrm{P_{II}}}(y+\phi, \mathrm{supp}\, \phi)
\end{equation}
for all $\phi\in C^\infty_0(\R^2,\R^2)$.  This notion of minimality is standard for many problems in which the energy of a localized solution is actually infinite due to non compactness of the domain.  The minimality of the solution of  \eqref{pain}  arising from the limit in Theorem \ref{theorem 1} (ii) is a direct consequence of the proof in Section \ref{proofs}.

Regarding Theorem \ref{theorem 1} (iii) we note that since the degree of the local limit of the rescaled global minimizer in $|x|>\rho$ is a  function whose topological degree is $1$ one may expect that the zero level set of $v_{\epsilon, a}$ is non empty and that isolated zeros correspond  to topological defects which should locally resemble  the well known Ginzburg-Landau vortices. We will show that this is partly true as non standard vortices occur in the physical regime of parameters. 


Before stating our second result we introduce  the standard Ginzburg-Landau vortex of degree one which is the radially symmetric solution of
\begin{equation}\label{likegl}
\Delta \eta=(|\eta|^2-1)\eta,\qquad \ \eta:\R^2\to\R^2,
\end{equation}
such that $\eta(x)=\eta_{\mathrm{rad}}(|x|)\frac{x}{|x|}$. 
We say that   $u$ is a minimal solution of (\ref{likegl}) if
\[
E_{\mathrm{GL}}(u, \mathrm{supp}\, \phi)\leq E_{\mathrm{GL}}(u+\phi, \mathrm{supp}\, \phi),
\]
for all $\phi\in C^\infty_0(\R^2,\R^2)$, where 
\[
E_{\mathrm{GL}}(u,\Omega):=\int_\Omega \frac{1}{2}|\nabla u|^2+\frac{1}{4}(1-|u|^2)
\]
is the Ginzburg-Landau energy associated to \eqref{likegl}. It is known \cite{MR1267609} that any minimal solution of (\ref{likegl}) is either constant of modulus $1$ or has degree $\pm 1$.
Mironescu  \cite{mironescu} showed moreover 
that any minimal solution of  \eqref{likegl} is either a constant of modulus $1$ or 
(up to orthogonal transformation in the range and translation in the domain) the radial solution $\eta$. We also mention some properties of $\eta$:
\begin{itemize}
\item[(i)] $\eta_{\mathrm{rad}}'>0$ on $(0,\infty)$, $\eta_{\mathrm{rad}}(0)=0$, $\lim_{r\to\infty} \eta_{\mathrm{rad}}(r)=1$,
\item[(ii)] $\int_{\R^2}|\nabla \eta|^2=\infty$.
\end{itemize}

Our next theorem shows existence of topological defects of the global minimizer of $E$  in several regimes of the parameters $(\epsilon, a)$:
\begin{theorem}\label{theoremz}
Assume that $a(\epsilon)>0$, $a$ is bounded and $\lim_{\epsilon\to 0}\epsilon^{1-\frac{3\gamma}{2}}\ln a=0$ for some $\gamma \in [0,2/3)$. 
\begin{itemize}
\item[(i)]  For $\epsilon\ll 1$,
 the global minimizer $v_{\epsilon, a}$ has at least one zero $\bar x_\epsilon$ such that 
\begin{equation}\label{zest}
|\bar x_ \epsilon|\leq \rho+o(\epsilon^\gamma).
\end{equation}
In addition, any sequence of zeros of $v_{\epsilon, a}$, either satisfies \eqref{zest} or it diverges to $\infty$. 
\item[(ii)] For every $\rho_0\in(0,\rho)$, there exists $b_*>0$ such that when $\limsup_{\epsilon\to 0}\frac{a}{\epsilon|\ln\epsilon|}<b_*$  then any limit point $l\in \R^2$ of the set of zeros of the global minimizer satisfies 
\begin{equation}\label{zestbbb}
\rho_0\leq|l|\leq\rho.
\end{equation}
In addition if $a=o(\epsilon|\ln\epsilon|)$ then $|l|=\rho$. 
\item[(iii)] On the other hand, for every $\rho_0\in(0,\rho)$, there exists $ b^*>0$ such that when 
$\limsup_{\epsilon\to 0}\frac{a}{\epsilon|\ln\epsilon|^2}>b^*$,
 the set of zeros of the global minimizer has a limit point $l$ such that 
\begin{equation}\label{zestccc}
|l|\leq\rho_0.
\end{equation}
If $v_{\epsilon, a}(\bar x_\epsilon)=0$  and $\bar x_\ve\to l$ then up to a subsequence
\[
\lim_{\epsilon\to 0} v_{\epsilon, a}(\bar x_\epsilon+\epsilon s)\to \sqrt{\mu(l)}(g\circ\eta)(\sqrt{\mu(l)} s),
\]
in $C^2_{\mathrm{loc}}(\R^2)$, for some $g\in O(2)$. 
In addition if $\limsup_{\epsilon\to 0}\frac{a}{\epsilon|\ln\epsilon|^2}=\infty$ then $l=0$.
\end{itemize}
\end{theorem}
To discuss physical consequences of this theorem we state: 
\begin{theorem}\label{th1n}
\begin{itemize}
 \item[(i)] When $a=0$ the global minimizer can be written as $v(x)=(v_{\mathrm{rad}}(|x|),0)$ with $v_{\mathrm{rad}} \in C^\infty(\R)$ positive. 
It is unique up to change of $v$ by $gv$ with $g \in SO(2)$.
\item[(ii)] Given $\epsilon>0$, there exists $A>0$ such that for every $a>A$, the global minimizer $v_{\epsilon,a}$ is unique and radial
i.e. $v(x)=v_{\mathrm{rad}}(|x|)\frac{x}{|x|}$.
\end{itemize}
\end{theorem}
Actually radial minimizers such as in Theorem \ref{th1n} (ii) exist for all $(\ve, a)$ with $\epsilon>0$ and $a\geq 0$ . Indeed it can be shown that 
 in the class $H^1_{\mathrm{rad}}(\R^2,\R^2):=\{u \in H^1(\R^2,\R^2): gu(x)=u(gx),\ \forall g \in O(2 ) \}$ of radial maps 
(or $O(2)$-equivariant maps), there exists $u \in H^1_{\mathrm{rad}}(\R^2,\R^2)$ such that $E(u)=\min_{H^1_{\mathrm{rad}}(\R^2,\R^2)} E$. Existence of the radial minimizer $u$ follows as in the proof of Lemma \ref{lem exist min}, and clearly $u$ is a critical point of $E$ in the subspace $H^1_{\mathrm{rad}}(\R^2,\R^2)$. 
In view of the radial symmetry of $\mu$ and $f$, one can show that the Euler-Lagrange equation \eqref{euler} 
holds for every $\phi \in H^1(\R^2,\R^2)$ (cf. \cite{palais}). As a consequence, 
$u(x)=u_{\mathrm{rad}}(|x|)\frac{x}{|x|}$ is a $C^\infty$ classical solution of \eqref{ode}. In addition, 
proceeding as in the proof of Theorem \ref{th1n}, 
it is easy to see that the radial minimizer is unique and satisfies $u_{\mathrm{rad}}>0$ on $(0,\infty)$ for every $\epsilon>0$ and $a>0$.

Theorem \ref{th1n} shows that when $a=0$ the global minimizer of $E$ inherits the one dimensional radial profile of $\mu$.
On the other hand it would be natural to expect that when $a>0$ the forcing term $\epsilon a f$ in \eqref{ode} induces a global minimizer $v\in H^1_{\mathrm{rad}}$.
Theorem \ref{theoremz} shows that this is not the case.
Indeed, according to the statement (ii) of Theorem \ref{theoremz} the symmetry the global minimizer is not radial   as soon as $\limsup_{\epsilon\to 0}\frac{a}{\epsilon|\ln\epsilon|}<b_*$, since this condition implies that no limit point of the zeros of the global minimizers belongs to $D(0;\rho_0)$ (cf. Lemma~\ref{cvz}). We point out that the hypothesis $\lim_{\epsilon\to 0}\epsilon^{1-\frac{3\gamma}{2}}\ln a=0$ for some $\gamma \in [0,2/3)$ was assumed in Theorem~\ref{theoremz} only to ensure the existence of a sequence of zeros satisfying \eqref{zest}, in particular the assertion of Lemma  \ref{cvz} remains valid for any $a$. 
Theorem \ref{theoremz} (iii) states further increase of the value of $a$ leads to the restoration of the symmetry at least in the limit 
$\epsilon\to 0$. 
Finally, Theorem \ref{th1n} (ii) shows that the symmetry is completely restored provided that $a$ is large enough.

The energy $E$ belongs to the class of Ginzburg-Landau type functionals that appear for example  in the theory of superconductivity or in the theory of Bose-Einstein condensates (see for instance \cite{book:971664}, \cite{MR1707887}, \cite{MR1696100}, \cite{doi:10.1142/S0129055X00000411}, \cite{MR1731999},  \cite{serfaty1}, \cite{aftalion6}, \cite{MR2186426}, \cite{MR2062641}, \cite{Ignat2006260}, \cite{Millot_energyexpansion} and the references therein). The Gross-Pitaevskii energy functional appearing in the latter theory has form
\[
E_{\mathrm{GP}}(u)=\int_{\R^2} \frac{1}{2}|\nabla u|^2 +\frac{1}{2\ve^2} V(x)|u|^2+\frac{1}{4\ve} |u|^4- \Omega x^\perp\cdot  (iu, \nabla u) \quad \mbox{subject to}\quad \|u\|_{L^2}=1 ,
\] 
where $\Omega\in \R$ is the angular velocity,  $ (iu, \nabla u)=i u\nabla \bar u -i\bar u \nabla u$ and $V(x)= x_1 +\Lambda x_2$ is a harmonic 
trapping potential (more general  nonnegative, smooth $V$ are considered as well).  The relation between $E_{\mathrm{GP}}$ and $E$ can be understood   if  we
recast the Gross-Pitaevskii energy  taking into account the mass  constraint in the form 
\begin{equation}
E_{\mathrm{GP}}(u)=\int_{\R^2} \frac{1}{2}|\nabla u|^2 +\frac{1}{4\ve^2}\left[\left(|u|^2-a(x)\right)^2-\left(a^-(x)\right)^2\right]^2- \Omega x^\perp\cdot  (iu, \nabla u), 
\label{gp 1}
\end{equation}
where $a(x)=a_0-V(x)$, $a_0$ is determined so that $\int_{\R^2} a^+=1$ and $a^\pm$ 
are the positive and negative parts of $a$. 
The angular velocity has certain threshold values at which different global minimizers appear. 
When  $\Omega=\mathcal O(|\ln \ve|)$ is below a certain critical value $\Omega_1$  
global minimizers are   vortex free \cite{Ignat2006260,MR2772375, {MR3355003}}, while at some 
other critical values $\Omega_2>\Omega_1$ global minimizers have at least one vortex \cite{Ignat2006260,Millot_energyexpansion}, which looks locally like the radially 
symmetric degree $\pm 1$ solution to the Ginzburg-Landau equation (\ref{likegl})
These localized structures 
have  analogues for the energy functional $E$: when $a=0$ the global minimizer is a vortex free state and when $a\sim \epsilon|\ln\epsilon|^2$ the global minimizer has one vortex that looks like the standard Ginzburg-Landau vortex (see Figure \ref{Fig-3 vortices} (a)). Possible qualitative difference   between the two functionals is manifested in the intermediate region for the values of $a$.  When $a$ satisfies the hypothesis of Theorem \ref{theoremz} (ii)  the global minimizer has a vortex which however can not be easily associated with the standard vortex (see Figure \ref{Fig-3 vortices} (c)). Based on numerical simulations we conjecture that, rather than coming from the equation (\ref{likegl}),  its rescaled  local profile   comes from the generalized second Painlev\'e equation  (\ref{pain}). We call this new type of defect the shadow vortex (the name is inspired from the physical context, see \cite{panayotis_1}).  Note that the amplitude of the shadow vortex is very small, of order $O(\epsilon^{1/3})$, in contrast with the standard vortex whose amplitude is of order $O(1)$. 
Numerical simulations show that there exists standard vortex minimizers localized at $|\bar x_\epsilon|=\rho_0$ strictly between $0$ and $\rho$ --- this happens when $a\sim \epsilon|\ln\epsilon|$. 
Despite the similarities between our model and the Gross-Pitaevskii functional it is not clear whether the shadow vortex exists for the Bose-Einstein condensate --- proving  this is  a delicate matter because, unlike the energy of the standard vortex which is of order $|\ln\epsilon|$,  the energy of the shadow vortex is relatively small.

\begin{figure}
\includegraphics[width=15cm,height=8cm]{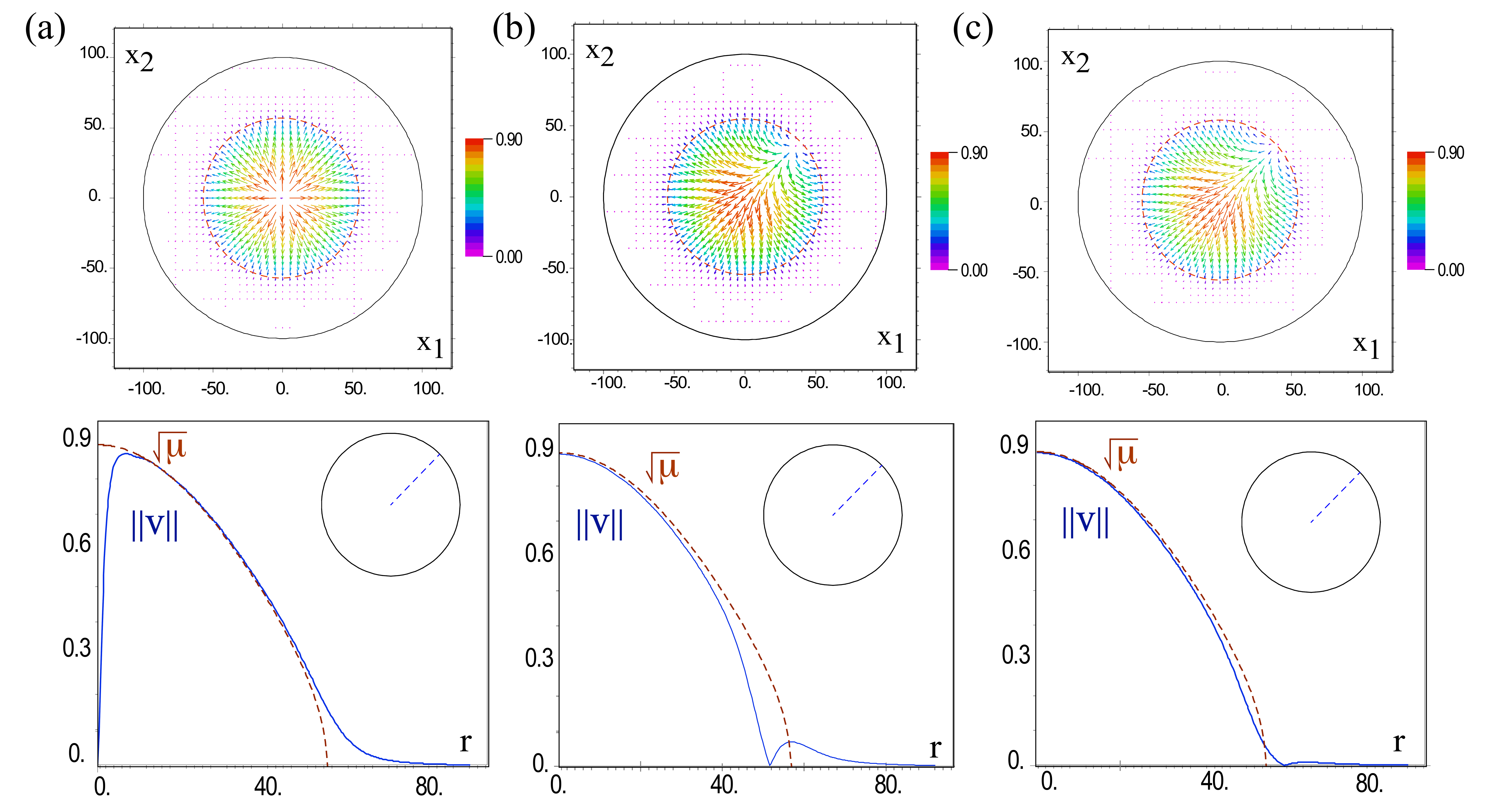}
\caption{(a) The standard vortex described in Theorem \ref{theoremz} (iii)   (b) The standard vortex near the boundary of the set $\mu>0$ and (c) The shadow vortex  described in Theorem \ref{theoremz}  (ii) in case $a=o(\epsilon|\ln\epsilon|)$. The upper panel shows the global minimizer $v=(v_1, v_2)$ as a vector field in $\R^2$. In the lower panel a radial section of $|v|$ taken at the angle $\theta$ indicated in the upper right corner and compared with the Thomas-Fermi limit $\sqrt{\mu^+}$. Numerical simulations where performed after   rescaling the original  spacial variable $x\mapsto x/\epsilon$.}
\label{Fig-3 vortices}
\end{figure} 
%

The symmetry breaking scenario described above can be seen from another angle since the shadow vortex can be interpreted as a transient  vortex state between the homogenous state and the standard vortex state as $a$ is increasing. In the Ginzburg-Landau theory of superconductivity the onset of vortex state is associated with the hysteresis phenomenon near the lower critical field where the energy of the non-vortex state (Meissner solution) equals that of the single vortex state \cite{doi:10.1137/S0036141096298060}.  The difference  with  the case considered here seems  due to the non smoothness of the Thomas-Fermi limit and the mediating effect of the solution of the Painlev\'e equation---in essence it is a boundary layer phenomenon. Still, the results of numerical simulations shown in Figure \ref{Fig-3 vortices} suggest that the shadow vortex may exist and be locally stable beyond  $a=o(\epsilon|\ln \epsilon|)$ and that the critical value of $a$ when the global minimizer becomes the standard vortex occurs when its the energy and that of the shadow vortex are equal. This would point out to the presence of hysteresis also in our case.   

This paper is organized as follows: in the next section we establish existence and basic properties of the global minimizer and in  section \ref{proofs} we prove our theorems.

\section{General results for minimizers and solutions}
In this section we gather general results for minimizers and solutions that are valid for any values of the parameters $\epsilon> 0$ and $a\geq 0$. We first prove the existence of global minimizers.
\begin{lemma}\label{lem exist min}
For every $\epsilon> 0$ and $a\geq 0$, there exists $v \in H^1(\R^2,\R^2)$ such that $E(v)=\min_{H^1(\R^2,\R^2)} E$. As a consequence, $v$ is a $C^\infty$ classical solution of \eqref{ode}, 
and moreover 
$v(x)\to 0$ as $|x|\to \infty$.
\end{lemma}
\begin{proof}
We first show that $\inf\{\,E(u):\ u \in H^1(\R^2,\R^2) \, \}>-\infty$. 
To see this, we regroup the last three terms in the integral of $E(u)$. 
Setting $I_\delta:=\{ x \in \R^2: \mu(x)+\delta>0\}$, for $\delta>0$ sufficiently small such that $I_\delta$ is bounded, we have
$$ -\frac{1}{2\epsilon^2}\mu(x)|u|^2+\frac{1}{8\epsilon^2}|u|^4< 0 \Longleftrightarrow u^2< 4\mu \Longrightarrow x\in I_{\delta},$$
thus
$$ -\frac{1}{2\epsilon^2}\mu(x)|u|^2+\frac{1}{8\epsilon^2}|u|^4\geq -\frac{2}{\epsilon^2}\left\| \mu\right\|^2_{L^\infty} \chi_\delta,$$
where $\chi_\delta$ is the characteristic function of $I_\delta$. On the other hand,
$$\frac{1}{8\epsilon^2}|u|^4-\frac{a}{\epsilon} f(x)\cdot u<0\Longrightarrow |u|^3\leq 8 a\epsilon|f|\Longrightarrow|fu|\leq  (8a\epsilon)^{1/3}|f|^{4/3},$$
thus
$$\frac{1}{8\epsilon^2}|u|^4-\frac{a}{\epsilon} f(x)\cdot u\geq-\frac{ 8^{1/3}a^{4/3}}{\epsilon^{2/3}}|f|^{4/3}.$$
Next, we notice that $E(u)\in \R$ for every $u \in H^1(\R^2,\R^2)$, thanks to the imbedding $H^1(\R^2)\subset L^p(\R^2)$, for $2\leq p <\infty$. 
Now, let $m:=\inf_{H^1}E>-\infty$, and let $u_n$ be a sequence such that $E(u_n) \to m$. Repeating the previous computation, we can bound
\begin{align}
\int_{\R^2}\frac{1}{2}| \nabla u_n|^2+\frac{\delta}{2\epsilon^2}|u_n|^2 &=E(u_n)+\int_{\R^2}\frac{1}{2\epsilon^2}(\mu(x)+\delta)|u_n|^2-\frac{1}{4\epsilon^2}|u_n|^4+\frac{a}{\epsilon} f(x)\cdot u_n\nonumber \\
&\leq E(u_n)+\frac{2}{\epsilon^2}(\left\| \mu\right\|_{L^\infty}+\delta)^2 |I_\delta|+\frac{ 8^{1/3}a^{4/3}}{\epsilon^{2/3}}\int_{\R^2}|f|^{4/3}.\nonumber
\end{align}
From this expression it follows that $\left\|u_n \right\|_{H^{1}(\R^2,\R^2)}$ is bounded. 
As a consequence, for a subsequence still called $u_n$, $u_n\rightharpoonup v$ weakly in $H^1$, and thanks to a diagonal argument we also have 
$u_n\to v$ in $L^2_{\mathrm{loc}}$, and almost everywhere in $\R^2$. Finally, by lower semicontinuity 
$$\int_{\R^2} |\nabla v|^2\leq \liminf_{n \to \infty} \int_{\R^2}  |\nabla u_n|^2,$$ and by Fatou's Lemma we have 
$$\int_{\R^2} |v|^4\leq \liminf_{n \to \infty} \int_{\R^2}  |u_n|^4, \text{ and }\int_{\mu\leq 0}-\frac{1}{2\epsilon^2}\mu |v|^2 \leq \liminf_{n \to \infty} \int_{\mu\leq 0}-\frac{1}{2\epsilon^2}\mu |u_n|^2.$$ 
To conclude, it is clear that $$\int_{\mu> 0}-\frac{1}{2\epsilon^2}\mu |v|^2= \lim_{n \to \infty} \int_{\mu > 0}-\frac{1}{2\epsilon^2}\mu |u_n|^2,$$ thus $m\leq E(v)\leq\liminf_{n \to \infty}  E(u_n)=m$.
Next, we check that $v$ is bounded. This follows from the fact that there exists a constant $M$ such that for every $x\in\R^2$ and $i=1,2$ 
the function
$$u_i\to -\frac{1}{2\epsilon^2}\mu(x)|u|^2+\frac{1}{4\epsilon^2}|u|^4-\frac{a}{\epsilon} f(x)\cdot u$$ is strictly increasing on $[M,\infty)$ 
(resp. strictly decreasing on $(\infty,-M]$)
independently of the other variable $u_j$ ($j\neq i$, $j=1,2$). Thus, if we truncate a map $u=(u_1,u_2)$ by setting $\tilde u_i=\min(M,\max(u_i,-M))$, 
the truncated map $\tilde u$ has smaller energy than $u$.
Clearly the boundedness of $v$ implies by \eqref{ode} the boundedness of $\Delta v$, and $\nabla v$. In particular, $v$ and $|v|^4$ are uniformly continuous. 
As a consequence if $|v(x_n)|>\delta>0$ for a sequence $|x_n|\to\infty$, 
then we would have $|v|>\delta/2$ on a ball $B(x_n,r)$ of radius $r$ independent of $n$, and also $\int_{\R^2}|v|^4=\infty$, which is impossible. 
This proves the asymptotic convergence of $v$ to $0$.  
\end{proof}

In the sequel, we will always denote the global minimizer by $v$.
%
To study the limit of solutions as $\epsilon\to 0$, we need to establish uniform bounds in the different regions considered in Theorem \ref{theorem 1}:
\begin{lemma}\label{s3}
For $\epsilon a$ belonging to a bounded interval, let $u_{\epsilon,a}$ be a solution of \eqref{ode} converging to $0$ as $|x|\to \infty$. Then, the solutions $u_{\epsilon,a}$  and the maps $\epsilon\nabla u_{\epsilon,a}$ are uniformly bounded. 
\end{lemma}

\begin{proof}
We drop the indexes and write $u:=u_{\epsilon,a}$.
Since $|f|$, $\mu$, and $\epsilon a$ are bounded, the roots of the cubic equation in the variable $u_1$
$$u_1^3+(u_2^2-\mu(x))u_1-\epsilon a f_1(x)=0$$ belong to a bounded interval, for all values of $x$, $u_2$, $\epsilon$, $a$. 
If $u_1$ takes positive values, then it attains its maximum $0\leq \max_{\R^2}u_1=u_1(x_0)$, at a point $x_0\in\R^2$. 
In view of \eqref{ode}: $$0\geq\epsilon^2 \Delta u_1(x_0)=u_1^3(x_0)+(u_2^2(x_0)-\mu(x_0))u_1(x_0)-\epsilon a f_1(x_0),$$ thus it follows that 
$u_1(x_0)$ is uniformly bounded above. In the same way, we prove the uniform lower bound for $u_1$, and the uniform bound for $u_2$. 
The boundedness of $\epsilon\nabla u_{\epsilon,a}$ follows from \eqref{ode} and the uniform bound of  $u_{\epsilon,a}$.
\end{proof}

\begin{lemma}\label{l2}
For $\epsilon\ll 1$ and $a$ belonging to a bounded interval, let $u_{\epsilon,a}$ be a solution of \eqref{ode} converging to $0$ as $|x|\to\infty$. Then, there exist a constant  $K>0$ such that 
\begin{equation}\label{boundd}
|u_{\ve,a}(x)|\leq K(\sqrt{\max(\mu (x),0)}+\ve^{1/3}), \quad \forall x\in \R^2.
\end{equation}
As a consequence, if for every $\xi=\rho e^{i\theta}$ we consider the local coordinates $s=(s_1,s_2)$ in the basis $(e^{i\theta},i e^{i\theta})$, then the rescaled maps $\tilde u_{\epsilon,a}(s)=\frac{u_{\ve,a}(\xi+ s\epsilon^{2/3})}{\epsilon^{1/3}}$ are uniformly bounded on the half-planes $[s_0,\infty)\times\R$, $\forall s_0\in\R$. 
\end{lemma}
\begin{proof}
For the sake of simplicity we drop the indexes and write $u:=u_{\epsilon,a}$. Let us define the following constants
\begin{itemize}
\item $ M>0$ is the uniform bound of $|u_{\epsilon,a}|$ (cf. Lemma \ref{s3}),
\item $\lambda>0$ is such that $3 \mu_{\mathrm{rad}}(\rho-h)\leq 2\lambda h$, $\forall h \in [0,\rho]$,
\item $F:=\sup_{\R^2}| f|$,
\item $\kappa>0$ is such that $\kappa^3\geq 3aF$, and $\kappa^4\geq 6\lambda$.
\end{itemize}
Next, we construct the following comparison function
\begin{equation}\label{compchi}
\chi(x)=
\begin{cases}
\lambda\Big(\rho-|x|+\frac{\epsilon^{2/3}}{2}\Big)&\text{ for } |x|\leq\rho,\\
\frac{\lambda}{2\epsilon^{2/3}}(|x|-\rho-\epsilon^{2/3})^2&\text{ for }\rho\leq |x|\leq\rho+\epsilon^{2/3},\\
0&\text{ for } |x|\geq\rho+\epsilon^{2/3}.
\end{cases}
\end{equation}
One can check that $\chi\in C^1(\R^2\setminus\{0\})\cap H^1(\R^2)$ satisfies $\Delta \chi\leq \frac{2\lambda}{\epsilon^{2/3}}$ in $H^1(\R^2)$. Finally, we define the function
$\psi:=\frac{|u|^2}{2}-\chi-\kappa^2\epsilon^{2/3}$, and compute:
\begin{align}
\epsilon^2 \Delta \psi&=\epsilon^2 (|\nabla u|^2+ u\cdot\Delta u-\Delta \chi)\nonumber\\
&\geq-\mu |u|^2+|u|^4-\epsilon a f\cdot u-\epsilon^2 \Delta \chi \nonumber\\
&\geq-\mu |u|^2+|u|^4-\epsilon a F| u|-2\epsilon^{4/3}\lambda.
\end{align}
Now, one can see that when $x\in \omega:=\{x\in\R^2: \psi(x)> 0\}$, we have $\frac{|u|^4}{3}- \mu|u|^2\geq 0$, since
$$x \in\omega\cap \overline{D(0;\rho)}\Rightarrow \frac{|u|^4}{3}\geq \frac{2\lambda}{3}\Big(\rho-|x|+\frac{\epsilon^{2/3}}{2}\Big)|u|^2\geq \mu|u|^2 .$$
On the open set $\omega$, we also have: $\frac{|u|^4}{3}\geq \frac{\kappa^4}{3}\epsilon^{4/3}\geq 2\epsilon^{4/3}\lambda$, 
and $\frac{|u|^4}{3}\geq \frac{\kappa^3}{3}\epsilon|u|\geq \epsilon aF|u|$. Thus $\Delta \psi \geq 0$ on $\omega$ in the $H^1$ sense. 
To conclude, we apply Kato's inequality that gives: $\Delta \psi^+ \geq 0$ on $\R^2$ in the $H^1$ sense. Since $\psi^+$ is subharmonic with compact
support, we obtain by the maximum principle that $\psi^+\equiv 0$ or equivalently $\psi \leq 0$ on $\R^2$. The statement of the lemma follows by adjusting the constant $K$.
\end{proof}



\begin{lemma}\label{s3gg}
Assume that $a$ is bounded and let $u_{\epsilon,a}$ be solutions of \eqref{ode} uniformly bounded. Then, the maps $\frac{u_{\epsilon,a}}{\epsilon}$ and $\nabla u_{\epsilon,a}$ are uniformly bounded on the sets $\{x:\, |x|\geq \rho_1\}$ for every $\rho_1>\rho$. 
\end{lemma}

\begin{proof}
We consider the sets $S:= \{ x:\  |x|\geq\rho_1\}\subset S':= \{ x:\  |x|>\rho'_1\}$, with $\rho<\rho'_1<\rho_1$, and define the constants:
\begin{itemize}
\item $ M>0$ which is the uniform bound of $|u_{\epsilon,a}|$,
\item $ \mu_0=-\mu_{\mathrm{rad}}(\rho'_1)>0$, 
\item $ f_\infty=\|f\|_{L^\infty}$,
\item $a^*:=\sup a(\epsilon)$,
\item $k=\frac{2a_*f_\infty}{\mu_0}>0$.
\end{itemize}
Next we introduce the function $\psi(x)=\frac{1}{2}(|u|^2-k^2\epsilon^2)$ satisfying:
\begin{align*}
\epsilon^2 \Delta \psi=\epsilon^2 \Delta\frac{|u|^2}{2}&\geq |u|^4+\mu_0| u|^2-\epsilon a_*f_\infty |u|  \ , \forall x\in S',\\
&\geq \mu_0 \psi , \ \forall x \in S' \text{ such that }  \psi(x)\geq 0.
\end{align*}
By Kato's inequality we have $\epsilon^2\Delta \psi^+ \geq  \mu_0\psi^+$ on $S'$, in the $H^1$ sense, and utilizing a standard 
comparison argument, we deduce that $\psi^+(x)\leq M^2 e^{-\frac{c}{\epsilon}d(x,\partial S')}$, $\forall x \in S$, and $
\forall \epsilon\ll 1$, where $d$ stands for the Euclidean distance, and $c>0$ is a constant. It is clear that 
$$d(x,\partial S')>-\frac{\epsilon}{c}\ln\Big(\frac{k^2
\epsilon^2}{2 M^2}\Big)\Rightarrow M^2e^{-\frac{c}{\epsilon}d(x,\partial S')}<\frac{k^2\epsilon^2}{2}\Rightarrow |u|^2<2k^2\epsilon^2.$$
Therefore, there exists $\epsilon_0$ such that 
\begin{equation}\label{asd1}
\frac{|u_{\epsilon,a}(x)|}{\epsilon}\leq \sqrt{2} k,\ \forall \epsilon<\epsilon_0,\ \forall x\in S.
\end{equation}
The boundedness of $\nabla u_{\epsilon,a}$ follows from \eqref{ode} and the uniform bound \eqref{asd1}.
\end{proof}

\section{Proof of Theorems \ref{theorem 1}, \ref{theoremz} and \ref{th1n}}\label{proofs}

\begin{proof}[Proof of Theorem \ref{theorem 1} (i)] 


Suppose by contradiction that $|v|$ does not converge uniformly to $\sqrt{\mu}$ on a closed set $F\subset \Omega$. Then there exist a sequence $\epsilon_n\to 0$
and a sequence $\{x_n\}\subset F$ such that
\begin{equation}\label{either}
\text{either $|v_{\epsilon_n}(x_n)|\geq \sqrt{\mu(x_n)}+\delta$ or $|v_{\epsilon_n}(x_n)|\leq \sqrt{\mu(x_n)}-\delta$, for some $\delta>0$}.
\end{equation}
In addition, we may assume that up to a subsequence $\lim_{n\to\infty}x_n=x_0\in F$.
Next, we consider the rescaled maps
$\tilde v_n(s)=v_{\epsilon_n}(x_n+\epsilon_n s)$ that satisfy
\begin{equation}\label{asd2ccee}
\Delta \tilde v(s)+\mu( x_n+\epsilon_n s)\tilde v(s)-|\tilde v(s)|^2\tilde v(s)+\epsilon_n af( x_n+\epsilon_n s)=0 , \ \forall s \in \R^2.
\end{equation}
In view of the Lemma \ref{s3} and \eqref{asd2ccee}, $\tilde v_{n}$ and its first derivatives are uniformly bounded for $\epsilon \ll 1$. 
Moreover, by differentiating \eqref{asd2ccee}, one also obtains the boundedness of the second derivatives of $\tilde v_n$ on compact sets.
Thus, we can apply the theorem of Ascoli via a diagonal argument, and show that for a subsequence still called $\tilde v_n$,
$\tilde v_n$ converges in $C^2_{\mathrm{ loc}}(\R^2,\R^2)$ to a map $\tilde V$, that we are now going to determine.
For this purpose, we introduce the rescaled energy
\begin{equation}
\label{functresee}
\tilde E(\tilde u)=\int_{\R^2}\Big(\frac{1}{2}|\nabla \tilde u(s)|^2-\frac{1}{2}\mu( x_n+\epsilon_n s)|\tilde u(s)|^2+\frac{1}{4}|\tilde u(s)|^4-\epsilon_n a f( x_n+\epsilon_n s)\cdot\tilde u(s)\Big)\dd s=E(u),\nonumber
\end{equation}
where we have set $\tilde u(s)=u_{\epsilon_n}( x_n+\epsilon_ns)$ i.e. $u_{\epsilon_n}(x)=\tilde u\big(\frac{x-x_n}{\epsilon_n}\big)$.
Let $\tilde \xi$ be a test function with support in the compact set $K$. We have $\tilde E(\tilde v_n+\tilde \xi,K)\geq \tilde E(\tilde v_n,K)$, and at the limit
$G_{0}( \tilde V+\tilde\xi,K)\geq  G_{0}(\tilde V,K)$, where $$ G_{0}(\psi,K)=\int_{K}\left[\frac{1}{2}|\nabla\psi|^2-\frac{1}{2}\mu(x_0)|\psi|^2+\frac{1}{4}|\psi |^4\right],$$
or equivalently $G( \tilde V+\tilde\xi,K)\geq G(\tilde V,K)$, where
\begin{equation}\label{glee}
G(\psi,K)=\int_{K}\left[\frac{1}{2}|\nabla\psi|^2-\frac{1}{2}\mu(x_0)|\psi|^2+\frac{1}{4}|\psi|^4+\frac{(\mu(x_0))^2}{4}\right]
=\int_{K}\left[\frac{1}{2}|\nabla\psi|^2+\frac{1}{4}(|\psi|^2-\mu(x_0))^2\right].
\end{equation}
Thus, we deduce that $\tilde V$ is a bounded minimal solution of the P.D.E. associated to the functional \eqref{glee}:
\begin{equation}\label{odegl}
\Delta \tilde V(s)+(\mu(x_0)-|\tilde V(s)|^2)\tilde V(s)=0.
\end{equation}
If $\tilde V$ is a constant of modulus $\sqrt{\mu(x_0)}$, then we have $\lim_{n\to\infty}|v_{\epsilon_n}(x_n)|=\sqrt{\mu(x_0)}$ which is excluded by \eqref{either}.
Therefore we obtain (up to orthogonal transformation in the range) $\tilde V(s)=\sqrt{\mu(x_0)}\,\eta(\sqrt{\mu(x_0)}(s-s_0))$, where 
$\eta(s)=\eta_{\mathrm{rad}}(|s|)\frac{s}{|s|}$ is the radial solution to the Ginzburg-Landau equation \eqref{likegl}, and $s_0\in \R^2$. In particular,
the degree of $\tilde V$ on $\partial D(0; 2|s_0|)$ is $\pm 1$, 
and by the $C^1_{\mathrm{ loc}}(\R^2,\R^2)$ convergence, we deduce that for $\epsilon_n\ll 1$ the degree of $\tilde v_n $ is still $\pm 1$ on
$\partial D(0; 2|s_0|)$. This implies that $v_n $ has a zero in $D(x_n; 2\epsilon_n|s_0|)$ for $\epsilon_n\ll 1 $, which contradicts the fact that $v_\epsilon\neq 0$ on $\Omega$ for $\epsilon\ll 1$.
\end{proof}
\begin{proof}[Proof Theorem \ref{theorem 1} (ii)] 
For every $\xi=\rho e^{i\theta}$ we consider the local coordinates $s=(s_1,s_2)$ in the basis $(e^{i\theta},i e^{i\theta})$,
and we rescale the global minimizers $v$ as in Lemma \ref{l2} by setting $\tilde v_{\epsilon,a}(s)=\frac{v_{\epsilon,a}(\xi+ s\epsilon^{2/3})}{\epsilon^{1/3}}$.
Clearly $\Delta v(s)=\epsilon \Delta v(\xi+s\epsilon^{2/3})$, thus,
\begin{equation}\label{oderes1}
\Delta \tilde v(s)+\frac{\mu(\xi+s\epsilon^{2/3})}{\epsilon^{2/3}} \tilde v(s)-|\tilde v^2(s)|\tilde v(s)+ a f(\xi+s\epsilon^{2/3})=0, \qquad \forall s\in \R^2.\nonumber
\end{equation}
Writing $\mu(\xi+h)=\mu_1 h_1+h\cdot A(h)$, with $\mu_1:=\mu'_{\mathrm{rad}}(\rho)<0$, $A \in C^\infty(\R^2,\R^2)$, and $A(0)=0$, we obtain
\begin{equation}\label{oderes2}
\Delta \tilde v(s)+(\mu_1 s_1 + A(s \epsilon^{2/3})\cdot s) \tilde v(s)-|\tilde v^2(s)|\tilde v(s)+ a f(\xi+s\epsilon^{2/3})=0,\qquad  \forall s\in \R^2.
\end{equation}
Next, we define the rescaled energy by
\begin{equation}
\label{functres2}
\tilde E(\tilde u)=\int_{\R}\Big(\frac{1}{2}|\nabla\tilde u(s)|^2-\frac{\mu(\xi+s \epsilon^{2/3})}{2\epsilon^{2/3}}\tilde u^2(s)+\frac{1}{4}|\tilde u|^4(s)- a f(\xi+s \epsilon^{2/3})\cdot \tilde u(s)\Big)\dd s.
\end{equation}
With this definition $\tilde E(\tilde u)=\frac{1}{\epsilon^{2/3}}E(u)$.
From Lemma \ref{l2} and \eqref{oderes2}, it follows that $\Delta \tilde v$, and also $\nabla\tilde v$, are uniformly bounded on compact sets. Moreover, by differentiating \eqref{oderes2} we also obtain the boundedness of the second derivatives of $\tilde v$.
Thanks to these uniform bounds, we can apply the theorem of Ascoli via a diagonal argument to obtain the convergence of $\tilde v$ in $C^2_{\mathrm{ loc}}(\R^2,\R^2)$ (up to a subsequence) 
 to a minimal solution (cf. \eqref{minnn}) $\tilde V $ of the P.D.E. 
\begin{equation}\label{oderes4}
\Delta \tilde V(s)+\mu_1 s_1 \tilde V(s)-|\tilde V|^2(s)\tilde V(s)+ a_0 f(\xi)=0, \ \forall s\in \R^2, \text{ with } a_0:=\lim_{\epsilon\to 0}a(\epsilon),
\end{equation}
which is associated to the functional
\begin{equation}
\label{functres4}
\tilde E_0(\phi,J)=\int_{J}\Big(\frac{1}{2}|\nabla\phi(s)|^2-\frac{\mu_1}{2} s_1 |\phi|^2(s)+\frac{1}{4}|\phi|^4(s)- a_0 f(\xi)\cdot\phi(s) \Big)\dd s.
\end{equation}
Setting $y(s):=\frac{1}{\sqrt{2}(-\mu_1)^{1/3}}\tilde V\big(\frac{s}{(-\mu_1)^{1/3}}\big)$, \eqref{oderes4} reduces to \eqref{pain} with 
$\alpha=\frac{a_0f(\xi)}{\sqrt{2}\mu_1}$, and $y$ is still a minimal solution of \eqref{pain}. In addition, by Lemma \ref{boundd}, $\tilde V$ and $y$ are bounded in the half-planes $[s_0,\infty)\times\R$, $\forall s_0\in\R$. 
\end{proof}
\begin{proof}[Theorem \ref{theorem 1} (iii)] 
For every $x_0=r_0 e^{i\theta_0}$ fixed, with $r_0>\rho$, we consider the local coordinates $s=(s_1,s_2)$ in the basis $(e^{i\theta_0},i e^{i\theta_0})$, and the rescaled maps
$\tilde v_{\epsilon,a}(s)= \frac{v_{\epsilon,a}(x_0+\epsilon s)}{\epsilon}$, satisfying 
\begin{equation}\label{asd2}
\Delta \tilde v(s)+\mu(x_0+\epsilon s)\tilde v(s)-\epsilon^2|\tilde v(s)|^2\tilde v(s)+af(x_0+\epsilon s)=0 , \ \forall s \in \R^2.
\end{equation}
In view of the bound \eqref{asd1} provided by Lemma \ref{s3gg} and \eqref{asd2}, we can see that the first derivatives of $\tilde v_{\epsilon,a}$ are uniformly bounded on compact sets for $\epsilon \ll 1$. 
Moreover, by differentiating \eqref{asd2}, one can also obtain the boundedness of the second derivatives of $\tilde v$ on compact sets.
As a consequence, we conclude that $\lim_{\epsilon\to 0, a\to a_0}\tilde v_{\epsilon,a}(s)=\tilde V(s)$ in $C^2_{\mathrm{loc}}$, where $\tilde V(s)\equiv-\frac{a_0}{\mu_{\mathrm{rad}}(r_0)}f(r_0 e^{i\theta_0})$ is the unique bounded solution of
\begin{equation}\label{asd3}
\Delta \tilde V(s)+\mu(x_0)\tilde V(s)+a_0f(x_0)=0 , \ \forall s \in \R^2.
\end{equation}
Indeed, consider a smooth and bounded solution $u:\R^2\to\R^2$ of $\Delta u=\nabla W(u)$ where the potential $W:\R^2\to \R$ is smooth and strictly convex. Then, we have $\Delta(W(u))=|\nabla W(u)|^2+\sum_{i=1}^2D^2W(u)(u_{x_i},u_{x_i})\geq 0$, and since $W(u)$ is bounded we deduce that $W(u)$ is constant. Therefore, $u\equiv u_0$ where $u_0\in \R^2$ is such that $\nabla W(u_0)=0$. 
Finally, the uniform convergence $\lim_{\epsilon\to 0, a\to a_0}\frac{u_{\epsilon,a}((r_0+t\epsilon)e^{i\theta})}{\epsilon}=-\frac{a_0}{\mu_{\mathrm{rad}}(r_0)}f(r_0 e^{i\theta})$, when $t$ remains bounded and $\theta\in\R$, follows from the invariance of equation \eqref{ode} under the transformations $u(x)\mapsto g^{-1}u(gx)$, $\forall g \in SO(2)$.
\end{proof}

\begin{proof}[Theorem \ref{theoremz} (i)] 
The proof follows from the next lemma which applies in the more general case of uniformly bounded solutions:
\begin{lemma}\label{s3g}
Consider the annulus $A=\{\rho+\tau \epsilon^{\gamma}\leq |x|\leq \rho_1\}$ with $\rho_1>\rho$, $\tau>0$ and $\gamma\in [0,2/3)$ fixed. 
Assume that $a(\epsilon)>0$ and $\lim_{\epsilon\to 0}\epsilon^{1-\frac{3\gamma}{2}}\ln( a)=0$, and 
let $u_{\epsilon,a}$ be solutions of \eqref{ode} uniformly bounded. Then, there exists $\epsilon_0$ such that 
$u_{\epsilon,a}(x)\neq 0$, $\forall x \in A$, $\forall \epsilon<\epsilon_0$. In addition,
\begin{itemize}
\item $\lim_{\epsilon\to 0}\frac{u_{\epsilon,a}(x)}{|u_{\epsilon,a}(x)|}=\frac{x}{|x|}$, uniformly on $A$, 
\item when $\epsilon \ll 1$, the solution $u_{\epsilon,a}$ has at least one zero $\bar x_\epsilon$ in the open disc $|x|<\rho+\tau\epsilon^\gamma$.
\end{itemize}
\end{lemma}

\begin{proof}
We examine the sign of the projections $u_\nu(x)=- (u(x)\cdot \nu)$, where $\nu=(\cos\theta_0,\sin\theta_0)$ is a unit vector. 
Consider for every $\delta>0$ the set
$$S_\delta:=\left \{ x=re^{i\theta}:\ \rho+ \tau\epsilon^\gamma\leq r\leq\rho_1, \ -\frac{\pi}{2}+\delta+\theta_0\leq\theta\leq\frac{\pi}{2}-\delta+\theta_0\right\},$$
which is contained in the domain
$$S'_\delta:=\left \{ x=re^{i\theta}:\ \rho+\frac{\tau}{2}\epsilon^{\gamma}< r<\rho_2, \ -\frac{\pi}2+\frac{\delta}{2}+\theta_0<\theta<\frac{\pi}2-\frac{\delta}{2}+\theta_0\right\},$$
with $\rho_1<\rho_2$.
In view of \eqref{hyp2}, let $0<f_\delta:=\cos(\frac{\pi-\delta}{2})\min_{r\in [\rho,\rho_2]}f_{\mathrm{rad}}(r)\leq\inf_{x\in S'_\delta} (f(x)\cdot\nu)$, 
and notice that
$$\epsilon^2 \Delta u_\nu\geq (|u|^2-\mu) u_\nu+\epsilon a f_\delta, \ \forall x\in S'_\delta.$$
Next, we define:
\begin{itemize}
\item $ M>0$ which is the uniform bound of $|u_{\epsilon,a}|$,
\item $ \mu_0>0$ such that $2\mu_0 h\leq -\mu_{\mathrm{rad}}(\rho+h)$, for $h\in[0,1]$, 
\item $ \mu_\infty=\sup_{\R^2}(-\mu)>0$,
\item $k_\epsilon=\frac{af_\delta}{M^2+\mu_\infty}>0$,
\end{itemize}
and the function $\psi(x)=u_\nu+k_\epsilon\epsilon$. One can check that when $x\in \omega:=\{x\in S'_\delta: \psi(x)> 0\}$, we have
$\epsilon^2\Delta \psi \geq\tau\epsilon^{\gamma} \mu_0\psi$ on $\omega$. To extend the previous inequality to the domain $S'_\delta$, 
we apply Kato's inequality that gives: $\epsilon^{2-\gamma}\Delta \psi^+ \geq  \tau\mu_0\psi^+$ on $S'_\delta$, in the $H^1$ sense. 
Now, since in $\gamma\in[0,2/3)$, we can see that $\forall x \in S_\delta$: $d(x,S'_\delta)\geq \kappa \epsilon^\gamma\gg \epsilon^{1-\frac{\gamma}{2}}$ for some constant $\kappa>0$, where $d$ stands for the Euclidean distance,
and utilizing a standard 
comparison argument, we deduce that $\psi^+(x)\leq (M +k_\epsilon \epsilon)e^{-\frac{c}{\epsilon^{1-\frac{\gamma}{2}}}d(x,\partial S'_\delta)}$, $\forall x \in S_\delta$, 
$\forall \epsilon\ll 1$, where $c>0$ is a constant. 
Finally, in view of $\lim_{\epsilon\to 0}\epsilon^{1-\frac{3\gamma}{2}}\ln( a)=0$ and $\gamma\in[0,2/3)$, there exists $\epsilon_\delta$ (independent of $\theta_0$) such that 
\begin{equation}\label{qss}
\forall\epsilon<\epsilon_\delta, \ \forall x\in S_\delta: \ d(x,\partial S'_\delta)>-\frac{\epsilon^{1-\frac{\gamma}{2}}}{c}\ln\Big(\frac{k_\epsilon\epsilon}{ M+k_\epsilon\epsilon}\Big)\Rightarrow \psi^+(x)\leq (M+k_\epsilon\epsilon) e^{-\frac{c}{\epsilon^{1-\frac{\gamma}{2}}}d(x,\partial S'_\delta)}<k_\epsilon\epsilon.\end{equation}
From this it follows  $u_\nu(x)<0$ hence $u(x)\cdot\nu>0$

To conclude, we notice that every $x=re^{i\theta}\in A$ belongs to the intersection of the sets $S_\delta$ corresponding to the angles $\theta_0\in [\theta-\frac{\pi}{2}+\delta,\theta+\frac{\pi}{2}-\delta]$. 
As a consequence, $\forall x=re^{i\theta}\in A$, $\forall\epsilon<\epsilon_\delta$, $\forall \theta_0\in [\theta-\frac{\pi}{2}+\delta,\theta+\frac{\pi}{2}-\delta]$, we have $u(x)\cdot (\cos\theta_0,\sin\theta_0)>0$,
and in particular
\begin{itemize}
\item $u(x)\neq 0$,  
\item $u(x)=|u(x)|e^{i(\phi+\theta)}$, with $\phi\in (-\delta,\delta)$.  
\end{itemize}
Since for every $\delta>0$ arbitrary small, we can find an $\epsilon_\delta>0$ such that $\Big|\frac{u(x)}{|u(x)|}-\frac{x}{|x|}\Big|\leq|e^{i\phi}-1|$ holds 
$\forall x\in A$, $\forall \epsilon<\epsilon_\delta$, with $\phi\in (-\delta,\delta)$, it follows that
$\lim_{\epsilon\to 0}\frac{u_{\epsilon,a}(x)}{|u_{\epsilon,a}(x)|}=\frac{x}{|x|}$, uniformly on $A$.
In addition, for $\epsilon<\epsilon_\delta$ (with $\delta$ small), the winding number of $u$ on the circle $|x|=\rho+\tau\epsilon^\gamma$ is one. 
Thus by degree theory, the solution $u$ has at least one zero in the open disc $|x|<\rho+\tau\epsilon^\gamma$.
\end{proof}
\end{proof}
\begin{proof}[Theorem \ref{theoremz} (ii)] 
The minimum of the energy defined in \eqref{funct 0} is nonpositive and tends to $-\infty$ as $\epsilon \to 0$. Since we are interested in the behavior of the minimizers as $\epsilon \to 0$, it is useful to define a renormalized energy, which is obtained by adding to \eqref{funct 0} a suitable term so that the result is tightly bounded from above. We define the renormalized energy as 
\begin{equation}\label{renorm}
\mathcal{E}(u):=E(u)+\int_{|x|<\rho}\frac{\mu^2}{4\epsilon^2}=\int_{\R^2}\frac{1}{2}|\nabla u|^2+\int_{|x|<\rho}
\frac{(|u|^2-\mu)^2}{4\epsilon^2}+\int_{|x|>\rho}\frac{|u|^2(|u|^2-2\mu)}{4\epsilon^2}- \frac{a}{\epsilon} \int_{\R^2}f\cdot u ,
\end{equation}  
and claim the bound:
\begin{lemma}\label{bbnm}
\begin{equation}\label{quaq1}
\mathcal{E}(v_{\epsilon,a})\leq 
\frac{\pi|\mu_1|\rho}{6}|\ln\epsilon|+\mathcal O(1) \text{ for $\epsilon\ll 1$ and arbitrary $a$},
\end{equation}
where $\mu_1=\mu_{\mathrm{rad}}'(\rho)$.
\end{lemma}
\begin{proof}
Let us consider the $C^1$ piecewise map $\psi=(\psi_1,\psi_2)$:
\begin{equation}
\psi_1(x)=
\begin{cases}
\sqrt{\mu(x)} &\text{for } |x|\leq \rho-\epsilon^{2/3} \\
k_\epsilon \epsilon^{-1/3}(\rho-|x|)  &\text{for }\rho-\epsilon^{2/3} \leq|x|\leq \rho\\
0  &\text{for } |x|\geq \rho
\end{cases},\qquad \psi_2 (x)=0, \nonumber
\end{equation} 
with $k_\ve$ defined by $k_\epsilon \epsilon^{1/3}=\sqrt{\mu_{\mathrm{rad}}(\rho- \epsilon^{2/3})}\Longrightarrow k_\epsilon=\mathcal O(1)$.
Since $\psi \in H^1(\R^2,\R^2)$, it is clear that $\mathcal E(v)\leq \mathcal E(\psi)$. 
We check that $\mathcal E(\psi)=\frac{\pi|\mu_1|\rho}{6}|\ln\epsilon|+\mathcal O(1)$, since it is the sum of the following integrals:
$$\int_{\rho-\epsilon^{2/3}<|x|<\rho}\frac{(|\psi|^2-\mu)^2}{4\epsilon^2}=\mathcal O(1),$$
$$\int_{|x|>\rho-\epsilon^{2/3}}\frac{1}{2} |\nabla \psi_1|^2=\mathcal O(1),$$
$$\int_{|x|\leq\rho-\epsilon^{2/3}}\frac{1}{2}\frac{ |\mu'_{\mathrm{rad}}(|x|)|^2}{4\mu}=\frac{|\mu_1|}{8}\int_{|x|\leq\rho-\epsilon^{2/3}}\frac{1}{\rho-|x|}+\mathcal O(1)=\frac{\pi|\mu_1|\rho}{6}|\ln\epsilon|+\mathcal O(1).$$
\end{proof}
We also compute a lower bound of the renormalized energy when $\frac{a}{\epsilon|\ln\epsilon|}$ is bounded:
\begin{lemma}\label{cvzaa}
Assuming that $\frac{a}{\epsilon|\ln\epsilon|}$ is bounded, then
for every $\rho_0<\rho$:
\begin{equation}\label{quaq1aac}
\liminf_{\epsilon\to 0}\frac{1}{|\ln \epsilon|}\int_{\rho_0\leq |x|\leq \rho}\frac{1}{2}|\nabla  v_{\epsilon}|^2\geq 
\frac{\pi|\mu_1|\rho}{6},
\end{equation}
where $\mu_1=\mu_{\mathrm{rad}}'(\rho)$.
\end{lemma}
\begin{proof}
Let $\gamma \in (0,2/3)$ and $\Omega_\epsilon=\{x:\rho_0\leq |x|\leq \rho-\epsilon^\gamma\}$. The upper bound \eqref{quaq1} implies that
\begin{equation}\label{comb1}
\int_{\Omega_\epsilon}(|v_\epsilon|^2-\mu)^2=\mathcal O(\epsilon^2 |\ln \epsilon|).
\end{equation}
On the other hand we also have
\begin{equation}\label{comb2}
\int_{\Omega_\epsilon} \frac{1}{(|v_\epsilon|+\sqrt{\mu})^2}\leq \int_{\Omega_\epsilon}\frac{1}{\mu}=\mathcal O( |\ln \epsilon|).
\end{equation}
Combining \eqref{comb1} with \eqref{comb2}, and setting $\sigma:=|v_\epsilon|$, we obtain
\begin{equation}\label{comb3}
\int_{\Omega_\epsilon} |\sigma-\sqrt{\mu}|=\mathcal O( \epsilon|\ln \epsilon|).
\end{equation}
At this stage we compute a lower bound of the difference
\begin{align*}
\int_{\Omega_\epsilon} |\nabla\sigma|^2-\int_{\Omega_\epsilon}|\nabla \sqrt{\mu}|^2&=\int_{\Omega_\epsilon} |\nabla(\sigma- \sqrt{\mu})|^2+2\int_{\Omega_\epsilon}
\nabla(\sigma- \sqrt{\mu})\nabla \sqrt{\mu}\\
&\geq 2\int_{\Omega_\epsilon}(-\Delta \sqrt{\mu})(\sigma- \sqrt{\mu})+\frac{\mu'_{\mathrm{rad}}(\rho-\epsilon^\gamma)}{\sqrt{\mu_{\mathrm{rad}}(\rho-\epsilon^\gamma)}}\int_{|x|= \rho-\epsilon^\gamma}(\sigma- \sqrt{\mu})+\mathcal O(1).
\end{align*}
In view of \eqref{comb3} we have $\int_{\Omega_\epsilon}(-\Delta \sqrt{\mu})(\sigma- \sqrt{\mu})=\mathcal O(\epsilon^{1-\frac{3\gamma}{2}}|\ln \epsilon|)$ (since $|\Delta \sqrt{\mu}|=\mathcal O(\epsilon^{-\frac{3\gamma}{2}})$ on $\Omega_\epsilon$),
while $\int_{|x|= \rho-\epsilon^\gamma}(\sigma- \sqrt{\mu})=\mathcal O(\epsilon^{\frac{\gamma}{2}})$ by Lemma \ref{l2}. Therefore
\begin{equation}\label{gamm}
\int_{\rho_0\leq |x|\leq \rho}|\nabla  v_{\epsilon}|^2\geq \int_{\Omega_\epsilon} |\nabla v_\epsilon|^2\geq \int_{\Omega_\epsilon} |\nabla\sigma|^2\geq \int_{\Omega_\epsilon}|\nabla \sqrt{\mu}|^2+\mathcal O(1)=\frac{\pi|\mu_1|\rho\gamma}{2}|\ln \epsilon|+\mathcal O(1),
\end{equation}
and $\liminf_{\epsilon\to 0}\frac{1}{|\ln \epsilon|}\int_{\rho_0\leq |x|\leq \rho}\frac{1}{2}|\nabla  v_{\epsilon}|^2\geq 
\frac{\pi|\mu_1|\rho\gamma}{4}$.
Finally, letting $\gamma\to \frac{2}{3}$ we deduce \eqref{quaq1aac}.
\end{proof}
Now we are going to establish 
\begin{lemma}\label{cvz}
For every $\rho_0\in(0,\rho)$, there exists $b_*>0$ such that when $\limsup_{\epsilon\to 0}\frac{a}{\epsilon|\ln\epsilon|}<b_*$ the set of zeros of the global minimizers cannot have a limit point $l\in D(0;\rho_0)$.
\end{lemma}
The proof of Lemma \ref{cvz} proceeds by contradiction. Let $\{\bar x_\epsilon\}$ be a sequence of zeros of $v_{\epsilon,a}$. Assuming that $\bar x_\epsilon$ converges (up to a subsequence) to a point $x_0\in D(0;\rho_0)$, with $\rho_0<\rho$, we will obtain the bound
\begin{equation}\label{claimdisk1}
 \liminf_{\epsilon\to 0}\frac{1}{|\ln \epsilon|}\int_{D(0;\rho_0)}\Big[\frac{1}{2}|\nabla v_\epsilon|^2+
\frac{(|v_\epsilon|^2-\mu)^2}{4\epsilon^2}\Big]\geq\lambda>0,
\end{equation} 
which combined with \eqref{quaq1aac}, gives for $b\ll 1$ a lower bound of the renormalized energy bigger than the upper bound \eqref{quaq1}.
The limit in \eqref{claimdisk1} will follow from
\begin{lemma}\label{techn}
Let $0<\rho_0<\rho$, and $a(\epsilon)\leq b_0\epsilon|\ln \epsilon|$ for some $b_0>0$. Then there exist constants $\lambda>0$ and $C>0$, such that for every disc 
$D(x_0;r_0)\subset D(0;\rho_0)$ with $r_0\in[\epsilon,|\ln \epsilon|^{-1/2}]$, the condition
\begin{equation}\label{renormBB2}
\int_{\partial D(x_0;r_0)}\frac{1}{2}|\nabla v_\epsilon|^2+\int_{\partial D(x_0;r_0)}
\frac{(|v_\epsilon|^2-\mu)^2}{4\epsilon^2}\leq\frac{\lambda}{r_0}
\end{equation} 
implies the bound
\begin{equation}\label{renormBB3}
\mathcal{E}(v_\epsilon, D(x_0;r_0))\leq C.
\end{equation} 
\end{lemma}
\begin{proof}
 Let $\lambda$ be such that 
 \begin{equation}\label{chlam}
  \sqrt{\mu_{\mathrm{rad}}(\rho_0)} -\sqrt{ \frac{2\lambda}{ \pi\mu_{\mathrm{rad}}(\rho_0)} }- 
  \sqrt{4\pi\lambda}=\frac{1}{2}\sqrt{ \mu_{\mathrm{rad}} (\rho_0)}.  
 \end{equation}
We first utilize inequality \eqref{renormBB2} to bound $v_\epsilon$ in modulus and argument on $\partial D(x_0;r_0)$.
From $$\fint_{\partial D(x_0;r_0)}
\big||v_\epsilon|^2-\mu\big|^2\leq\frac{2\epsilon^2\lambda}{\pi r_0^2}\leq \frac{2\lambda}{\pi },$$ 
it follows that there exists $\theta_0\in\R$ such that
$\big||v_\epsilon(x_0+r_0\mathrm{e}^{i\theta_0})|^2-\mu(x_0+r_0\mathrm{e}^{i\theta_0})\big|^2\leq\frac{2\lambda}{\pi }$. Thus,
$$\Big||v_\epsilon(x_0+r_0\mathrm{e}^{i\theta_0})|-\sqrt{\mu(x_0+r_0\mathrm{e}^{i\theta_0})}\Big|
\leq\sqrt{\frac{2\lambda}{\pi \mu(x_0+r_0\mathrm{e}^{i\theta_0})}}\leq \sqrt{\frac{2\lambda}{\pi \mu_{\mathrm{rad}}(\rho_0)}}.$$
On the other hand the condition
$$\frac{1}{2r_0}\int_0^{2\pi}\Big|\frac{\partial v_\epsilon}{\partial \theta}(x_0+r_0\mathrm{e}^{i\theta})\Big|^2\dd \theta\leq 
\int_{\partial D(x_0;r_0)}\frac{1}{2}|\nabla v_\epsilon|^2\leq\frac{\lambda}{r_0},$$
implies that
$\int_0^{2\pi}\big|\frac{\partial v_\epsilon}{\partial \theta}(x_0+r_0\mathrm{e}^{i\theta})\big|\dd \theta\leq \sqrt{4\pi\lambda}$, and
$|v_\epsilon (x_0+r_0\mathrm{e}^{i\theta_2})-v_\epsilon (x_0+r_0\mathrm{e}^{i\theta_1})|\leq \sqrt{4\pi\lambda}$, for
$\theta_2\in[\theta_1,\theta_1+2\pi]$.

In view of \eqref{chlam}, we deduce that $v_\epsilon(x_0+r_0\mathrm{e}^{i\theta})=\sigma(\theta)\mathrm{e}^{i(\phi(\theta-\theta_0)+\phi_0)}$,
with $\sigma(\theta)\geq \sigma_0:=\frac{1}{2}\sqrt{ \mu_{\mathrm{rad}} (\rho_0)}$, $\phi(0)=0$, and
$|\phi|\leq\frac{\pi}{6}$. 
Indeed, we check that
$$\sigma(\theta)\geq \sigma(\theta_0)-|\sigma(\theta_0)-\sigma(\theta)|\geq \sqrt{\mu(x_0+r_0\mathrm{e}^{i\theta_0})} -\sqrt{ \frac{2\lambda}{ \pi\mu_{\mathrm{rad}}(\rho_0)} }- 
  \sqrt{4\pi\lambda}\geq \sigma_0, \ \forall \theta \in [\theta_0,\theta_0+2\pi],$$
  $$|\sin (\phi(\theta-\theta_0))|\leq\frac{\sqrt{4\pi\lambda}}{\sigma(\theta_0)}\leq 
  \frac{ \sqrt{4\pi\lambda}}{ \sqrt{\mu_{\mathrm{rad}}(\rho_0)} -\sqrt{ \frac{2\lambda}{ \pi\mu_{\mathrm{rad}}(\rho_0)} }}= 
  \frac{ \sqrt{4\pi\lambda}}{ \sqrt{ \frac{2\lambda}{ \pi\mu_{\mathrm{rad}}(\rho_0)}}+2\sqrt{4\pi\lambda} }\leq\frac{1}{2}, \ \forall \theta \in [\theta_0,\theta_0+2\pi].$$


Next we define the comparison map
\begin{equation}\label{compm}
u(x_0+r\mathrm{e}^{i\theta})=
\Big(\frac{r}{r_0}[\sigma(\theta)-\mu^{1/2}(x_0+r_0\mathrm{e}^{i\theta})]+\mu^{1/2}(x_0+r\mathrm{e}^{i\theta})\Big)
\mathrm{e}^{i\big(\frac{r}{r_0}\phi(\theta-\theta_0)+\phi_0\big)}, \ \forall r\in [0, r_0], \forall \theta \in \R.
\end{equation}
It is clear that $u$ is continuous on $\overline{ D(x_0;r_0)}$, and that $u\equiv v$ on $\partial D(x_0;r_0)$. We are going to check that $u \in H^1( D(x_0;r_0),\R^2)$, since actually $\int_{ D(x_0;r_0)}|\nabla  u|^2\leq C$, where $C$ is a positive constant depending only on $\mu$, $f$, $b_0$, $\rho_0$ and the uniform bound provided by Lemma \ref{s3}. In what follows it will be convenient to denote by $C$ such a constant that may vary from line to line. 
Indeed, we have
\begin{multline*}
\Big|\frac{\partial u(x_0+r\mathrm{e}^{i\theta})}{\partial r}\Big|^2=\Big|\frac{\sigma(\theta)-\mu^{1/2}(x_0+r_0\mathrm{e}^{i\theta})}{r_0}+\frac{\partial \mu^{1/2}(x_0+r\mathrm{e}^{i\theta})}{\partial r}\Big|^2\\+\frac{\phi^2(\theta-\theta_0)}{r_0^2}\Big|\frac{r}{r_0}[\sigma(\theta)-\mu^{1/2}(x_0+r_0\mathrm{e}^{i\theta})]+\mu^{1/2}(x_0+r\mathrm{e}^{i\theta})\Big|^2,\end{multline*}
\begin{multline*}
\frac{1}{r^2}\Big|\frac{\partial u(x_0+r\mathrm{e}^{i\theta})}{\partial \theta}\Big|^2=\Big|\frac{1}{r_0}\Big[\sigma'(\theta)-\frac{\partial \mu^{1/2}(x_0+r_0\mathrm{e}^{i\theta})}{\partial \theta}\Big]+\frac{1}{r}\frac{\partial \mu^{1/2}(x_0+r\mathrm{e}^{i\theta})}{\partial \theta}\Big|^2\\+\frac{|\phi'(\theta-\theta_0)|^2}{r_0^2}\Big|\frac{r}{r_0}[\sigma(\theta)-\mu^{1/2}(x_0+r_0\mathrm{e}^{i\theta})]+\mu^{1/2}(x_0+r\mathrm{e}^{i\theta})\Big|^2.\end{multline*}
Hence $\int_{ D(x_0;r_0)}\big|\frac{\partial u(x_0+r\mathrm{e}^{i\theta})}{\partial r}\big|^2\leq C$.
To obtain the bound $\int_{ D(x_0;r_0)}\frac{1}{r^2}\big|\frac{\partial u(x_0+r\mathrm{e}^{i\theta})}{\partial \theta}\big|^2\leq C$, we utilize 
\eqref{renormBB2} that gives
$\int_0^{2\pi} |\sigma'(\theta)|^2\dd \theta\leq 2\lambda$ and $\int_0^{2\pi} |\phi'(\theta-\theta_0)|^2\dd \theta\leq \frac{2\lambda}{\sigma_0^2}$.
Finally, from \eqref{compm} we can see that
$$\big| |u(x_0+r\mathrm{e}^{i\theta})|-\mu^{1/2}(x_0+r\mathrm{e}^{i\theta})\big|\leq |\sigma(\theta)-\mu^{1/2}(x_0+r_0\mathrm{e}^{i\theta})|,$$
$$\big| |u(x_0+r\mathrm{e}^{i\theta})|^2-\mu(x_0+r\mathrm{e}^{i\theta})\big|\leq |\sigma^2(\theta)-\mu(x_0+r_0\mathrm{e}^{i\theta})|\frac{\big| |u(x_0+r\mathrm{e}^{i\theta})|+\mu^{1/2}(x_0+r\mathrm{e}^{i\theta})\big|}{|\sigma(\theta)+\mu^{1/2}(x_0+r_0\mathrm{e}^{i\theta})|},$$
$$\big| |u(x_0+r\mathrm{e}^{i\theta})|^2-\mu(x_0+r\mathrm{e}^{i\theta})\big|^2\leq C\big| |v_\epsilon(x_0+r_0\mathrm{e}^{i\theta})|^2-\mu(x_0+r\mathrm{e}^{i\theta})\big|^2,$$
and since 
$\int_0^{2\pi}\frac{( |v_\epsilon(x_0+r_0\mathrm{e}^{i\theta})|^2-\mu(x_0+r\mathrm{e}^{i\theta}))^2}{4\epsilon^2}\leq \frac{\lambda}{r_0^2}$ by \eqref{renormBB2}, we deduce that
$\int_{ D(x_0;r_0)}\frac{(|u|^2-\mu)^2}{4\epsilon^2}\leq C$. On the other hand it is obvious that
$-\frac{a}{\epsilon}\int_{ D(x_0;r_0)}  f\cdot u \leq C$, thus we obtain by minimality of $v_\epsilon$:
\begin{equation*}\label{renormBB}
\mathcal{E}(v_\epsilon,  D(x_0;r_0))\leq \mathcal{E}(u, D(x_0;r_0))\leq C,
\end{equation*}  
which completes the proof. 
\end{proof}

\begin{proof}[Proof of Lemma \ref{cvz}]
We assume that $\sup\frac{a}{\epsilon|\ln\epsilon|}\leq b_0$, where $b_0$ is an arbitrary constant.
Suppose by contradiction that $\bar x_\epsilon$ converges (up to a subsequence) to a point $x_0\in D(0;\rho_0)$ (with $\rho_0<\rho$), and consider the rescaled maps 
$\tilde v_\epsilon(s)=v_\epsilon(\bar x_\epsilon+\epsilon s)$ satisfying
\begin{equation}\label{asd2cc}
\Delta \tilde v(s)+\mu(\bar x_\epsilon+\epsilon s)\tilde v(s)-|\tilde v(s)|^2\tilde v(s)+\epsilon af(\bar x_\epsilon+\epsilon s)=0 , \ \forall s \in \R^2.
\end{equation}
In view of the Lemma \ref{s3} and \eqref{asd2cc}, the first derivatives of $\tilde v_{\epsilon}$ are uniformly bounded for $\epsilon \ll 1$. 
Moreover, by differentiating \eqref{asd2cc}, one also obtains the boundedness of the second derivatives of $\tilde v$ on compact sets.
Thus, we can apply the theorem of Ascoli via a diagonal argument, and show that for a subsequence still called $\tilde v_\epsilon$,
$\tilde v_\epsilon$ converges in $C^2_{\mathrm{ loc}}(\R^2,\R^2)$ to a map $\tilde V$, that we are now going to determine.
For this purpose, we introduce the rescaled energy
\begin{equation}
\label{functres}
\tilde E(\tilde u)=\int_{\R^2}\Big(\frac{1}{2}|\nabla \tilde u(s)|^2-\frac{1}{2}\mu(\bar x_\epsilon+s\epsilon)|\tilde u(s)|^2+\frac{1}{4}|\tilde u(s)|^4-\epsilon a f(\bar x_\epsilon+s\epsilon)\cdot\tilde u(s)\Big)\dd s=E(u),\nonumber
\end{equation}
where we have set $\tilde u(s)=u_\epsilon(\bar x_\epsilon+s\epsilon)$ i.e. $u_\epsilon(x)=\tilde u\big(\frac{x-\bar x_\epsilon}{\epsilon}\big)$.
Let $\tilde \xi$ be a test function with support in the compact set $K$. We have $\tilde E(\tilde v_\epsilon+\tilde \xi,K)\geq \tilde E(\tilde v_\epsilon,K)$, and at the limit
$G_{0}( \tilde V+\tilde\xi,K)\geq  G_{0}(\tilde V,K)$, where $$ G_{0}(\psi,K)=\int_{K}\left[\frac{1}{2}|\nabla\psi|^2-\frac{1}{2}\mu(x_0)|\psi|^2+\frac{1}{4}|\psi |^4\right],$$
or equivalently $G( \tilde V+\tilde\xi,K)\geq G(\tilde V,K)$, where
\begin{equation}\label{gl}
G(\psi,K)=\int_{K}\left[\frac{1}{2}|\nabla\psi|^2-\frac{1}{2}\mu(x_0)|\psi|^2+\frac{1}{4}|\psi|^4+\frac{(\mu(x_0))^2}{4}\right]
=\int_{K}\left[\frac{1}{2}|\nabla\psi|^2+\frac{1}{4}(|\psi|^2-\mu(x_0))^2\right].
\end{equation}
Thus, we deduce that $\tilde V$ is a bounded minimal solution of the P.D.E. associated to the functional \eqref{gl}:
\begin{equation}\label{odegl}
\Delta \tilde V(s)+(\mu(x_0)-|\tilde V(s)|^2)\tilde V(s)=0,
\end{equation}
and since $\tilde V(0)=0$, we obtain (up to orthogonal transformation in the range) $\tilde V(s)=\sqrt{\mu(x_0)}\,\eta(\sqrt{\mu(x_0)}s)$, where 
$\eta(s)=\eta_{\mathrm{rad}}(|s|)\frac{s}{|s|}$ is the radial solution to the Ginzburg-Landau equation \eqref{likegl}.
It is known that $\int_{\R^2}|\nabla \eta|^2=\infty$. Therefore, if $q>1$ is such that
$$3C<\mu_{\mathrm{rad}}(\rho_0)\int_{ D(0;q\sqrt{ \mu_{\mathrm{rad}} (\rho_0)}) }|\nabla \eta|^2\leq
\mu(x_0)\int_{D(0;q\sqrt{\mu(x_0)})}|\nabla \eta|^2=\int_{D(0;q)}|\nabla \tilde V|^2,$$
where $C$ is the constant given in Lemma \ref{techn}, then for $\epsilon\leq\epsilon_0$ small enough, we have $
\frac{3C}{2}<\frac{1}{2}\int_{D(\bar x_\epsilon;q\epsilon)}|\nabla v_\epsilon|^2$. In addition, by taking $\delta>0$ sufficiently small, we can ensure that $\mathcal{E}(v_\epsilon, D(\bar x_\epsilon;r))>C$, for every $r\in [q\epsilon, \delta|\ln \epsilon|^{-1/2}]$, and every $\epsilon\leq \epsilon_0$. Next, applying Lemma \ref{techn}, we obtain for $r\in [q\epsilon, \delta|\ln \epsilon|^{-1/2}]$ and $\epsilon\leq\epsilon_0$ the inequality:
\begin{equation}\label{renormBB3}
\int_{\partial D(\bar x_\epsilon;r)}\frac{1}{2}|\nabla v_\epsilon|^2+\int_{\partial D(\bar x_\epsilon;r)}
\frac{(|v_\epsilon|^2-\mu)^2}{4\epsilon^2}>\frac{\lambda}{r}.
\end{equation} 
Finally, an integration of \eqref{renormBB3} gives
\begin{equation*}\label{renormBB4}
\int_{D(\bar x_\epsilon;\delta)}\frac{1}{2}|\nabla v_\epsilon|^2+\int_{ D(\bar x_\epsilon;\delta)}
\frac{(|v_\epsilon|^2-\mu)^2}{4\epsilon^2}\geq \int_{q\epsilon}^{\delta|\ln \epsilon|^{-1/2}}\frac{\lambda}{r}\dd r\geq \lambda|\ln \epsilon|-\frac{\lambda}{2}\ln(|\ln \epsilon|)+\mathcal O(1),
\end{equation*} 
from which \eqref{claimdisk1} follows. Combining \eqref{quaq1aac} with \eqref{claimdisk1}, we immediately see that the upper bound \eqref{quaq1} 
is violated when $\limsup_{\epsilon\to 0}\frac{a}{\epsilon|\ln\epsilon|}<b_*:=\min\big(\frac{\lambda}{M\|f\|_{L^1}},b_0\big)$, with $M\geq\|v_\epsilon\|_{L^\infty}$ (cf. Lemma \ref{s3}). Therefore the convergence of $\bar x_\epsilon$ to a point $x_0$ such that $|x_0|<\rho_0$ is exluded provided $\limsup_{\epsilon\to 0}\frac{a}{\epsilon|\ln\epsilon|}<b_*$. 
\end{proof}
To complete the proof we utilize part (i) of Theorem \ref{theoremz}, and deduce that any limit point $l\in\R^2$ of the set of zeros of the global minimizers satisfies \eqref{zestbbb}. 
When $a=o(\epsilon |\ln\epsilon|)$, it is immediate that $|l|=\rho$. 
\end{proof}

\begin{proof}[Theorem \ref{theoremz} (iii)] 
In the set $\Omega_v:=\{x\in\R^2: v(x)\neq 0\}$, we consider the polar form of $v$:
\begin{equation}\label{polar2}
v(x)=|v(x)| \frac{v(x)}{|v(x)|}=:\sigma (x) \n(x),\text{ where } \sigma(x):=|v(x)|, \ \n(x):=\frac{v(x)}{|v(x)|}.
\end{equation}
Setting 
\begin{equation}\label{enepol}
F(x)=
\begin{cases}
0&\text{when } v(x) =0, \smallskip \\
|\nabla \sigma(x)|^2   + \sigma^2 (x) |\nabla \n (x)|^2, &\text{when } v(x)\neq 0, \smallskip \\
\end{cases}
\end{equation}
we get (cf. \cite{evans-gariepy})
\begin{equation}\label{enepol2}
\int_{\R^2}|\nabla v(x)|^2\dd x=\int_{\R^2} F(x)\dd x=\int_{\Omega_v}\big(|\nabla \sigma(x)|^2   + \sigma^2 (x) |\nabla \n(x)|^2\big)\dd x.
\end{equation}
The next Lemma which is based on the previous decomposition, provides some information on the direction of the vector field $v$:
\begin{lemma}
 Assuming that $a$ is bounded and $\rho_1>\rho$, there exist a constant $K$, such that 
 \begin{equation}\label{inegf}
 \int_{\Omega_v\cap\{|x|\leq \rho_1\}}\Big(\frac{1}{2} \sigma^2 |\nabla \n|^2+\frac{ a}{\epsilon}|f|\sigma\Big[1-\frac{x}{|x|}\cdot \n\Big]\Big)\leq K |\ln \epsilon| \text{ for } \epsilon\ll 1.
 \end{equation}\end{lemma}
 \begin{proof}
We define the constants:
\begin{itemize}
\item $ M>0$ which is the uniform bound of $|v_{\epsilon,a}|$,
\item $ \rho_2 =\rho_1+1$.
\end{itemize}
Next writing $x=r\mathrm{e}^{i\theta}$, we consider the comparison map
\begin{equation}
\psi(r\mathrm{e}^{i\theta})=
\begin{cases}
\frac{r}{\epsilon}\sigma(\epsilon\mathrm{e}^{i\theta})\mathrm{e}^{i\theta} &\text{for } r \in [0,\epsilon], \\
\sigma(x)\mathrm{e}^{i\theta} &\text{for } r \in [\epsilon ,\rho_1],\\
\frac{r-\rho_1}{\rho_2-\rho_1}v(x)+\frac{\rho_2-r}{\rho_2-\rho_1}\sigma(x)\mathrm{e}^{i\theta}   &\text{for } r \in [\rho_1,\rho_2].
\end{cases} \nonumber
\end{equation}
It is clear that $\psi\in H^1(D(0;\rho_2),\R^2)$ and that $\psi\equiv v$ for $|x|=\rho_2$, thus 
\begin{equation}\label{diff}
 \mathcal E(v,D(0;\rho_2))-\mathcal  E(\psi,D(0;\rho_2))\leq 0.
\end{equation}
Since $\epsilon \nabla v$ is uniformly bounded on $\R^2$ (cf. Lemma \ref{s3}), and $\frac{v}{\epsilon}$ as well as $\nabla v$ are uniformly bounded on $\{|x|\geq \rho_1\}$ (cf. Lemma \ref{s3gg}), one can check that 
$$\mathcal E(v,D(0;\epsilon))-\mathcal  E(\psi,D(0;\epsilon))\leq K,$$
$$\mathcal E(v,\{\rho_1\leq|x|\leq \rho_2\})-\mathcal  E(\psi,\{\rho_1\leq|x|\leq \rho_2\}))\leq K,$$
where $K$ is a constant depending only on $\mu$, $f$, $\rho_1$ and the previous uniform bounds, that may vary from line to line. Therefore,
\begin{equation}\label{inegf1}
 \int_{\Omega_v\cap \{\epsilon\leq|x|\leq \rho_1\}}\Big(\frac{1}{2} (|\nabla v|^2-|\nabla\psi|^2)+\frac{ a}{\epsilon}(|f|\sigma-f\cdot v)\Big)\leq K,\end{equation}
and since $(|\nabla v(x)|^2-|\nabla\psi(x)|^2)=\sigma^2(x) |\nabla \n(x)|^2-\frac{\sigma^2(x)}{|x|^2}$ holds for $x\in \Omega_v\cap \{\epsilon\leq|x|\leq \rho_1\}$,  we deduce that
 \begin{equation}\label{inegf00}
 \int_{\Omega_v\cap \{\epsilon\leq|x|\leq \rho_1\}}\Big(\frac{1}{2} \sigma^2 |\nabla \n|^2+\frac{ a}{\epsilon}(|f|\sigma-f\cdot v)\Big)\leq \pi M^2 (|\ln \epsilon|+\ln \rho_1)+K \text{ for } \epsilon\ll 1.
 \end{equation}
Finally, \eqref{inegf} follows by combining \eqref{inegf00} with 
\begin{equation}\label{inegf000}
 \int_{\Omega_v\cap \{|x|\leq \epsilon\}}\Big(\frac{1}{2} \sigma^2 |\nabla \n|^2+\frac{ a}{\epsilon}(|f|\sigma-f\cdot v)\Big)\leq K,
 \end{equation}
and adjusting the constant $K$.
 \end{proof}
Now we prove
\begin{lemma}\label{lemze}
For every $\rho_0\in(0,\rho)$, there exists $ b^*>0$ such that when $a$ is bounded and
$\limsup_{\epsilon\to 0}\frac{a}{\epsilon|\ln\epsilon|^2}>b^*$,
the zeros of the global minimizers have a limit point $l$ such that
$|l|\leq \rho_0$. In particular, the condition
$\limsup_{\epsilon\to 0}\frac{a}{\epsilon|\ln\epsilon|^2}=\infty$ with $a$ bounded, implies the existence of a zero $\bar x_\epsilon \to 0$.
\end{lemma}
\begin{proof}
Assume by contradiction that the zeros of the global minimizer $v_{\epsilon}$ have no limit point such that $|l|\leq \rho_0$. As a consequence, there exists $\epsilon_0>0$ such that $\forall \epsilon<\epsilon_0$, $\forall x\in \overline {D(0;\rho_0)}$: $v_\epsilon(x)\neq 0$. Moreover, proceeding as in the proof of Theorem \ref{theorem 1} (i), we can see that $|v_\epsilon|$ converges uniformly on $\overline {D(0;\rho_0)}$ to $\sqrt{\mu}$, as $\epsilon \to 0$. Thus, for $\epsilon\ll 1$ we have 
\[
\min_{\overline {D(0;\rho_0)}} \sigma^2_\epsilon \geq \frac{1}{2}\mu_{\mathrm{rad}}(\rho_0),
\]
and from
\eqref{inegf} we deduce that
 \begin{equation}\label{inegfbb}
 \int_{|x|\leq \rho_0} |\nabla \n|^2\leq K_1 |\ln \epsilon|, \  \ \frac{ a}{\epsilon}\int_{\rho_0/2\leq|x|\leq \rho_0}\Big[1-\frac{x}{|x|}\cdot \n\Big]\leq K_2 |\ln \epsilon| \text{ for } \epsilon\ll 1,
 \end{equation}
where $K_i$ ($i=1,2$) are constants. At this stage we notice that since for every $\epsilon<\epsilon_0$, $v_\epsilon$ does 
not vanish on $\overline {D(0;\rho_0)}$, the degree of $v_\epsilon$ on the circles $|x|=r$, with $r\in (0,\rho_0]$, is zero. 
In particular, we can write $\n(r\mathrm{e}^{i\theta})=\mathrm{e}^{i\phi_r(\theta)}$, where $\phi_r:\R\to \R$ is a $2\pi$-periodic smooth function, for every $r\in (0,\rho_0]$. 
Now we define the measurable sets
\begin{itemize}
\item $F:=\left\{x: \frac{\rho_0}{2}\leq |x|\leq\rho_0, \, \frac{x}{|x|}\cdot \n\leq\frac{1}{2}\right\}=\left\{ x=r\mathrm{e}^{i\theta}: \frac{\rho_0}{2}\leq r\leq\rho_0, \ \phi_r(\theta)-\theta\notin (-\pi/3,\pi/3)\mod 2\pi\right\}$,
\item $F_r:=\{\theta\in[0,2\pi]: \phi_r(\theta)-\theta\notin(-\pi/3,\pi/3)\mod 2 \pi\}$, 
\item $R:=\left\{ r\in [\rho_0/2,\rho_0]: \mathcal L^1(F_r)< \frac{16 \epsilon}{a\rho_0^2}K_2 |\ln \epsilon|\right\}$,
\item  $R^c:= \left\{ r\in [\rho_0/2,\rho_0]: \mathcal L^1(F_r)\geq \frac{16 \epsilon}{a\rho_0^2}K_2 |\ln \epsilon|\right\}=[\rho_0/2,\rho_0]\setminus R$,
\end{itemize}where $\mathcal L^n$ denotes the $n$-dimensional  Lebesgue measure. 
It follows from these definitions and  \eqref{inegfbb} that
 \begin{equation}\label{inegfbbcc}
\frac{4}{\rho_0}K_2|\ln\epsilon|\mathcal L^1(R^c)\leq
\frac{a\rho_0}{4\epsilon}\int_{\rho_0/2}^{\rho_0}\mathcal L^1(F_r)\dd r\leq\frac{a}{2\epsilon}\mathcal L^2(F)\leq K_2 |\ln \epsilon|, \text{ for } \epsilon\ll 1,
 \end{equation}
thus $\mathcal L^1(R)\geq\frac{\rho_0}{4}$.
Moreover since $\phi_r$ is periodic, for every $r\in [\rho_0/2,\rho_0]$ there exists $\theta_1(r)\in \R$ and $\theta_2(r)\in (\theta_1(r),\theta_1(r)+2\pi)$, such that $\phi_r(\theta_1(r))-\theta_1(r)=\frac{5\pi}{3}$, $\phi_r(\theta_2(r))-\theta_2(r)=\frac{\pi}{3}$, and $\phi_r(\theta) -\theta \in (\frac{\pi}{3}, \frac{5\pi}{3})$ for $\theta\in(\theta_1(r),\theta_2(r))$. By definition of $F_r$, we also have $\theta_2(r)-\theta_1(r)\leq\mathcal L^1(F_r)$. Next using the Cauchy-Schwarz inequality we obtain
\[
\frac{(4\pi)^2}{3^2}=\left(\int_{\theta_1(r)}^{\theta_2(r)} (\phi'_r(\theta)-1)\dd\theta\right)^2\leq(\theta_2(r)-\theta_1(r)) \int_{\theta_1(r)}^{\theta_2(r)}|\phi'_r(\theta)-1|^2\dd\theta.
\]
As a consequence, 
$$\frac{(4\pi)^2}{3^2\mathcal L^1(F_r)}\leq\frac{(4\pi)^2}{3^2(\theta_2(r)-\theta_1(r))}\leq \int_{\theta_1(r)}^{\theta_2(r)}|\phi'_r(\theta)-1|^2\dd \theta\Rightarrow \frac{(4\pi)^2}{3^2\mathcal L^1(F_r)}-\frac{8\pi}{3}\leq \int_{\theta_1(r)}^{\theta_2(r)}|\phi'_r(\theta)|^2\dd \theta,$$
and
\begin{align}\label{asdf}
K_1|\ln\epsilon|&\geq \int_{\rho_0/2\leq|x|\leq \rho_0} |\nabla \n|^2\geq \int_{\rho_0/2}^{\rho_0}\frac{\dd r}{r}\int_0^{2\pi}|\n_\theta(r,\theta)|^2\dd \theta = \int_{\rho_0/2}^{\rho_0}\frac{\dd r}{r}\int_0^{2\pi}|\phi'_r(\theta)|^2\dd \theta \nonumber \\
&\geq\frac{1}{\rho_0}\int_R \frac{(4\pi)^2\dd r}{3^2\mathcal L^1(F_r)}-\frac{4\pi}{3}\geq\frac{(\pi\rho_0)^2a}{6^2  K_2\epsilon |\ln \epsilon|}-\frac{4\pi}{3}.
\end{align}
Now we can see that \eqref{asdf} is violated when $\limsup_{\epsilon\to 0}\frac{a}{\epsilon|\ln \epsilon|^2}>\frac{6^2K_1K_2}{(\pi\rho_0)^2}=:b^*$.
Therefore, we have proved the existence of a limit point $l\in \overline{D(0;\rho_0)}$ provided 
$\limsup_{\epsilon\to 0}\frac{a}{\epsilon|\ln\epsilon|^2}>b^*$ (with $a$ bounded). The previous argument also establishes the existence 
of a sequence $\bar x_\epsilon \to 0$
when $\limsup_{\epsilon\to 0}\frac{a}{\epsilon|\ln\epsilon|^2}=\infty$ with $a$ bounded. 
\end{proof}
Finally if $v_{\epsilon, a}(\bar x_\epsilon)=0$ and $\bar x_\ve\to l$, we consider the rescaled maps $\tilde v_\epsilon(s)=v_\epsilon(\bar x_\epsilon+\epsilon s)$ and proceeding as in the proof of Theorem \ref{theorem 1} (i)
we obtain up to subsequence
\[
\lim_{\epsilon\to 0} v_{\epsilon, a}(\bar x_\epsilon+\epsilon s)\to \sqrt{\mu(l)}(g\circ\eta)(\sqrt{\mu(l)} s),
\]
in $C^2_{\mathrm{loc}}(\R^2)$, for some $g\in O(2)$. 
\end{proof}

\begin{proof}[Proof of Theorem \ref{th1n} (i)]

We first notice that $v \not\equiv 0$ for $\epsilon \ll 1$. 
Indeed, by choosing a test function of the form $\psi=(\sqrt{\mu_{\mathrm{rad}}}, 0) \chi\big(\frac{\rho-r}{\epsilon^{2/3}}\big)$, with $\chi$ a cutoff function supported in the left half line  one can see that 
\[
E(\psi)\leq -\frac{1}{4\epsilon^2}\int_0^\rho \mu_{\mathrm{rad}}^2(r)\,r\dd r+O(|\ln\epsilon|)<0, \qquad  \epsilon\ll 1.
\] 
Let $x_0\in \R^2$ be such that $v(x_0)\neq 0$. Without loss of generality we may assume that $v(x_0)=(v_1(x_0), 0)$ is contained  in the open right half-plane  $P=\{x_1>0\}$. Next, consider $\tilde v=(|v_1|,v_2)$ which is another global minimizer and thus another solution.
Clearly, in a sufficiently small disc $D\subset P$ centered at $v(x_0)$ we have $v_1=|v_1|>0$, and as a consequence of the unique continuation principle (cf. \cite{sanada})
we deduce that $v\equiv \tilde v$ on $\R^2 \Rightarrow v(\R^2)\subset \overline P$. Since the same conclusion holds for 
any open half-plane containing $v(x_0)$, we also obtain $v(\R^2)\subset \{\lambda v(x_0): \ \lambda\geq0\}$. As a consequence we have $v=(v_1,0)$ with $v_1\geq 0$, and $\epsilon^2 \Delta v_1+\mu v_1-v_1^3=0$. By the maximum principle, it follows that $v_1>0$ since $v_1\not\equiv 0$.

Now to prove that $v_1$ is radial  consider  the reflection with respect to the line $x_1=0$.
We can check that $E(v,\{x_1>0\})=E(v,\{x_1<0\})$, since otherwise by even reflection we can construct a map in $H^1$ with energy smaller  than $v$. Thus, the map 
$\tilde v(x) = v(|x_1|,x_2)$ is also a minimizer, and since $\tilde v= v$ on $\{x_1>0\}$, it follows by unique continuation that 
$\tilde v\equiv v$ on $\R^2$. Repeating the same argument for any line of reflection, we deduce that $v_1$ is radial.
To complete the proof, it remains to show the uniqueness of $v$ up to rotations. Let $\tilde v=(\tilde v_1,0)$ be another global minimizer with $\tilde v_1>0$ 
and $\tilde v_1\not\equiv v_1$. 
Putting $\psi=v$ in \eqref{euler}: 
\begin{equation}\label{euler2}
\int_{\R^2} -\epsilon^2 |\nabla v|^2+\mu |v|^2-|v|^4=0\Rightarrow E(v)=-\int_{\R^2}\frac{1}{4\epsilon^2}|v|^4,
\end{equation}
we obtain an alternative expression of the energy that holds for every solution of \eqref{ode} belonging to $H^1$.
In particular, this formula implies that $v_1 $ and $\tilde v_1$ intersect for $|x|=r>0$. However, setting
\begin{equation*}
w(x)=\begin{cases}
      v(x) &\text{ for } |x|\leq r\\
       \tilde v(x) &\text{ for } |x|\geq r,
     \end{cases}
 \end{equation*}
we can see that $w$ is another global minimizer, and again by the unique continuation principle we have $w\equiv v\equiv\tilde v$. 
\end{proof}

\begin{proof}[Proof of Theorem \ref{th1n} (ii)]
We need first to establish the three Lemmas below.
\begin{lemma}\label{pst1}
If $u$ is a solution of \eqref{ode} belonging to $H^1(\R^2,\R^2)$, then for every $\psi \in H^1(\R^2,\R^2)$, we have
\begin{equation}\label{difference}
E(u+\psi)-E(u)=\int_{\R^2}\Big(\frac{1}{2}|\nabla \psi|^2+\frac{(|u|^2-\mu)}{2\epsilon^2}|\psi|^2+\frac{(|\psi|^2+2(u\cdot\psi))^2}{4\epsilon^2}\Big).
\end{equation}
\end{lemma}
\begin{proof}
The Euler-Lagrange equation \eqref{euler} gives
\begin{equation}\label{eqq1}
\int_{\R^2} \Big(\sum_{j=1,2}\nabla \psi_j\cdot\nabla u_j\Big)=\int_{\R^2}\Big(\frac{\mu}{\epsilon^2}\psi\cdot u -\frac{|u|^2 \psi\cdot u}{\epsilon^2}+\frac{a}{\epsilon} f\cdot \psi\Big).
\end{equation}
On the other hand, we have the identity
\begin{equation}\label{eqq2}
\int_{\R^2}\Big(\frac{1}{2} |\nabla \psi+\nabla u|^2+\frac{1}{2}|\nabla u|^2-\sum_{j=1,2}(\nabla \psi_j+\nabla u_j)\cdot\nabla u_j\Big)=\int_{\R^2}\frac{1}{2}|\nabla \psi|^2.
\end{equation} 
Adding \eqref{eqq1} and \eqref{eqq2}, we obtain
\begin{equation}\label{eqq3}
\int_{\R^2}\frac{1}{2}\Big( |\nabla(\psi+u)|^2-|\nabla u|^2\Big)=\int_{\R^2}\Big(\frac{1}{2}|\nabla \psi|^2+\frac{\mu}{\epsilon^2}\psi\cdot u-\frac{|u|^2\psi\cdot u}{\epsilon^2}+\frac{a}{\epsilon} f\cdot \psi\Big)=:B,
\end{equation} 
and thus
\begin{align}
E(u+\psi)-E(u)&=B+\int_{\R^2}\Big(-\frac{\mu}{2\epsilon^2}\big(|u+\psi|^2-|u|^2\big)+\frac{1}{4\epsilon^2}\big(|u+\psi|^4-|u|^4\big)-\frac{a}{\epsilon} f\cdot \psi\Big)\nonumber\\
&=\int_{\R^2}\Big(\frac{1}{2}|\nabla \psi|^2+\frac{(|u|^2-\mu)}{2\epsilon^2}|\psi|^2+\frac{(|\psi|^2+2(u\cdot\psi))^2}{4\epsilon^2}\Big).\nonumber
\end{align}
\end{proof}
\begin{lemma}\label{p4b}
For every $a> 0$ let $\mu_a:\R^2\to\R$ be a measurable function satisfying $\mu_a\leq \mu$, and $\lim_{a\to\infty} \mu_a = -\infty$ a.e., then given $\epsilon>0$ there exists $A>0$ such that for every $a >A$ we have
\begin{equation}
\label{Poincarebis}
\int_{\R^2}\mu_a|\psi|^2 < \epsilon^2 \int_{\R^2}|\nabla \psi|^2,\qquad  \forall \psi \in H^1(\R^2,\R^2), \, \psi \neq 0.
\end{equation}
\end{lemma}
\begin{proof}
By homogeneity, it is sufficient to prove \eqref{Poincarebis} for $\left\|\psi \right\|_{H^{1}}=1$. Suppose by contradiction that 
%
\eqref{Poincarebis} does not hold. Then there exist a constant $C_0>0$, a sequence $a_n\to\infty$, and a sequence $\psi_n \in H^1(\R^2,\R^2)$, with $\left\|\psi_n \right\|_{H^{1}}=1$, such that
\begin{equation}\label{Poincare2bis}
\int_{\R^2}\mu_{a_n}|\psi_n|^2 \geq C_0\int_{\R^2}|\nabla \psi_n|^2.
\end{equation}
Since $\left\|\psi_n \right\|_{H^{1}}$ is bounded, we can extract a subsequence, still called $\psi_n$, such that $\psi_n\rightharpoonup \Psi$ weakly in $H^1$, and $\psi_n\to \Psi$ in $L^2_{\mathrm{loc}}$. Writing
$$ C_0\int_{\R^2}|\nabla \psi_n|^2 \leq \int_{\R^2}\mu_{a_n}|\psi_n|^2 \leq  \int_{I_{\delta}}\max(\mu_{a_n},0)|\psi_n|^2,$$
where $I_{\delta}:=\{x\in\R^2: \mu(x)>-\delta\}$ and $\delta >0$ is small,
we see that  $\lim_{n \to\infty}\int_{\R^2}|\nabla \psi_n|^2=0$. This implies by lower semicontinuity that 
$\int_{\R^2} |\nabla\Psi|^2\leq \liminf \int_{\R^2}  |\nabla \psi_n|^2=0$, hence
$\Psi\equiv 0$. In addition, we have $\lim_{n \to\infty}\int_{\R^2}|\psi_n|^2=1$, $\lim_{n \to\infty}\int_{I_\delta}|\psi_n|^2=0$, and $\lim_{n \to\infty}\int_{\R^2\setminus I_\delta}|\psi_n|^2=1$. As a consequence, 
$$\int_{\R^2}\mu_{a_n}|\psi_n|^2\leq \left\|\mu \right\|_{L^{\infty}}\int_{I_\delta}|\psi_n|^2-\delta\int_{\R^2\setminus I_\delta}|\psi_n|^2,$$ and taking the limit we find that
$\int_{\R^2}\mu_{a_n}|\psi_n|^2\leq-\frac{\delta}{2}$, for $n$ big enough, which contradicts \eqref{Poincare2bis}. 
\end{proof}
\begin{lemma}\label{th4c}
For $\epsilon>0$ and $x_0\in\R^2$ fixed, the global minimizer satisfies
\begin{equation}
\lim_{ a\to \infty} a^{-1/3}v_{\epsilon,a}(x_0+a^{-1/3}s)=\frac{\epsilon^{1/3} f(x_0)}{|f(x_0)|^{2/3}},
\end{equation} for the $C^2_{\mathrm{ loc}}(\R^2,\R^2)$ convergence.
\end{lemma}
\begin{proof}
We consider the rescaled maps $\tilde v(s)=a^{-1/3}v(x_0+a^{-1/3}s)$, satisfying
\begin{equation}\label{oderesc}
\epsilon^2\Delta\tilde v(s)+a^{-2/3}\mu(x_0+a^{-1/3}s) \tilde v(s)-|\tilde v|^2(s)\tilde v(s)+\epsilon f(x_0+a^{-1/3}s)=0, \, \forall s\in \R^2.
\end{equation}
Repeating the arguments in the proof of Lemma \ref{s3} one can see that
when $\epsilon$ is fixed and $\frac{1}{a}$ remains bounded, the maps $\tilde v_{\epsilon,a}$ are uniformly bounded up to the second derivatives. Therefore proceeding as in the proof of Theorem \ref{theorem 1} (i) and (ii), we deduce the convergence of $\tilde v$ as $a\to\infty$ to the unique bounded solution of 
\begin{equation}\label{odeglc}
\epsilon^2\Delta\tilde V(s)-|\tilde V(s)|^2\tilde V(s)+\epsilon f(x_0)=0, \, \forall s\in\R^2,
\end{equation}
which is the constant $\tilde V\equiv\frac{\epsilon^{1/3} f(x_0)}{|f(x_0)|^{2/3}}$.
\end{proof}
Let $\epsilon>0$ be fixed and let $\mu_a:=\mu-|v|^2$, where $v:=v_{\epsilon,\alpha}$ is a global minimizer. By Lemma \ref{th4c}, we know that for every $x \neq 0$, $\mu_a(x)$ converges pointwise to $-\infty$, as $a\to \infty$. Thus, by \eqref{Poincarebis}, there exists $A>0$, such that for every $a>A$ we have
$$\int_{\R^2}\Big(\frac{1}{2}|\nabla \psi|^2+\frac{(|v|^2-\mu)}{2\epsilon^2}|\psi|^2\Big)>0, \qquad \forall \psi \in H^1(\R^2,\R^2), \quad \psi \neq 0,$$
and also $E(v+\psi)>E(v)$ in view of \eqref{difference}. In particular it follows that when $a>A$, the global minimizer is unique and radial, since $v\equiv g^{-1}vg$, $\forall g \in O(2)$. 
\end{proof}



\providecommand{\bysame}{\leavevmode\hbox to3em{\hrulefill}\thinspace}
\providecommand{\MR}{\relax\ifhmode\unskip\space\fi MR }
\providecommand{\MRhref}[2]{%
  \href{http://www.ams.org/mathscinet-getitem?mr=#1}{#2}
}
\providecommand{\href}[2]{#2}

\end{document}